\documentclass[10pt]{article}
\usepackage[latin1]{inputenc}
\usepackage{amsmath,amsthm,amssymb,amsfonts,amsxtra}
\usepackage{graphicx}

\newtheorem{theorem}{Theorem}[section]
\newtheorem*{theorem*}{Theorem}
\newtheorem{lem}{Lemma}[section]
\newtheorem{corollary}{Corollary}[section]

\newtheorem{definition}[theorem]{Definition}
\newtheorem*{assumption}{Assumption}

\newtheorem*{remark}{Remark}

\newcommand{\setR}{\mathbb{R}}
\newcommand{\setN}{\mathbb{N}}

\newcommand{\XXint}[3]{{\setbox0=\hbox{$#1{#2#3}{\int}$}
      \vcenter{\hbox{$#2#3$}}\kern-.5\wd0}}
\newcommand{\Xint}[1]{\mathchoice
    {\XXint\displaystyle\textstyle{#1}}%
    {\XXint\textstyle\scriptstyle{#1}}%
    {\XXint\scriptstyle\scriptscriptstyle{#1}}%
    {\XXint\scriptscriptstyle\scriptscriptstyle{#1}}%
    \!\int}
\newcommand{\dashint}{\mathop{\Xint-}}
\newcommand{\settmp}[2]{#1\{{#2}#1\}}
\newcommand{\set}[1]{\settmp{}{#1}}
\newcommand{\bigset}[1]{\settmp{\big}{#1}}

\newcommand{\biggset}[1]{\settmp{\bigg}{#1}}

\newcommand{\normtmp}[2]{#1\lVert{#2}#1\rVert}
\newcommand{\norm}[1]{\normtmp{}{#1}}

\newcommand{\abstmp}[2]{#1\lvert{#2}#1\rvert}
\newcommand{\abs}[1]{\abstmp{}{#1}}

\newcommand{\biggabs}[1]{\abstmp{\bigg}{#1}}

\newcommand{\hbs}{H^1(B_R,S^2)} 
\newcommand{\hos}{H^1(\Omega ,S^2)} 
\newcommand{\Div}{\operatorname{div}} 
\newcommand{\curl}{\operatorname{curl}}

\newcommand{\support}{\operatorname{supp}}

\newcommand{\ui}{\overline{u}}

\newcommand{\B}{\mathbf{B}}
\newcommand{\A}{\mathbf{A}}
\newcommand{\Poincare}{Poincar\'e}
\newcommand{\diag}{\operatorname{diag}}

\title{\sc Boundary regularity for minimizers of the micromagnetic energy functional\
 }
\author{\sc  Alexander Huber}

\date{}

\begin{document}

\maketitle

\begin{abstract}
Motivated by the construction of time-periodic solutions for the three-dimensional Landau-Lifshitz-Gilbert equation in the case of soft and small ferromagnetic particles, we investigate the regularity properties of minimizers of the micromagnetic energy functional at the boundary. In particular, we show that minimizers are regular 
provided the volume of the particle is sufficiently small. The approach uses a reflection construction at the boundary and an adaption of the well-known regularity theory for minimizing harmonic maps into spheres.
\end{abstract}

\section{Introduction and statement of the main result}
This work is motivated by the construction of time-periodic solutions for the three-dimensional Landau-Lifshitz-Gilbert equation in the case of soft and small ferromagnetic particles. \textquotedblleft{}Small\textquotedblright{} means that the diameter (or three-dimensional volume) of the ferromagnetic sample is sufficiently small and \textquotedblleft{}soft\textquotedblright{} refers to the case where no material anisotropy is considered. In our work \cite{alex_small} (see also \cite{alex_diss}), one of the crucial ingredients for the presented analysis is the fact that minimizers of the micromagnetic energy functional are regular up to the boundary in the small particle case. It is the aim of the paper at hand to present a proof of this fact which is of interest for its own sake.

The minimization problem under consideration reads as follows: Minimize
\begin{align*}
 E^\eta(u) = \int_{\Omega} \abs{\nabla u}^2 \, dx + \eta^2 \int_{\setR^3} \abs{H[u]}^2 \, dx
\end{align*}
among all functions $u \in H^1(\Omega,S^2) = \bigset{u \in H^1(\Omega,\setR^3) \, \big| \,  \abs{u}=1 \text{ almost everywhere}}$, where the so-called stray field $H[u] \in L^2(\setR^3,\setR^3)$ is the unique solution of
\begin{align*}
 \begin{array}{rl}
    \Div \big( H[u] + \chi_{\Omega} \, u \big) &= 0
    \\[-0,2cm]
    \curl H[u] &= 0
 \end{array} \quad \text{in } \setR^3 \,.
\end{align*}
Here, $\Omega \subset \setR^3$ is a bounded domain, and $\eta>0$ is a parameter representing the size of the particle. We remark that the boundary values of minimizers are completely determined by the minimization process since we have not imposed additional assumptions for the competing mappings at the boundary. Moreover, we want to mention that the above problem is a rescaled version of the micromagnetic minimization problem and $\eta$ plays the role of a scaling parameter (see \cite{alex_small}, \cite{alex_diss}).

Our aim is to prove the following theorem:
\begin{theorem*}[\textbf{Main result}]
  Let $\Omega \subset \setR^3$ be a bounded $C^{2,1}$-domain. There exist positive constants $\eta_0 = \eta_0(\Omega)$ and $C_0=C_0(\Omega)$ with the following property: If $u_\eta$ is a minimizer of $E^\eta$ on $H^1(\Omega,S^2)$ with parameter $0<\eta\le\eta_0$, then 
  \begin{align*}
   u_\eta \in H^2_N(\Omega,\setR^3) \cap C^{1,\gamma}(\overline{\Omega},\setR^3) \qquad \text{and} \qquad \norm{\nabla u_\eta}_{L^\infty} \le C_0 \eta
  \end{align*}
  for every $\gamma \in (0,1)$, where \textquotedblleft$N$\textquotedblright stands for homogeneous Neumann boundary conditions, i.e.,  $\frac{\partial u_\eta}{\partial \nu} = 0$ on the boundary $\partial \Omega$ of $\Omega$ with outer normal $\nu$.
\end{theorem*}
Regularity questions for the above problem have already been studied in the papers by Hardt, Kinderlehrer \cite{hardtkinderlehrer}  
and Carbou \cite{carbou}. The idea in both papers is to consider 
the non-local term $H[u]$ in $E^\eta$ as a lower order perturbation of the 
Dirichlet energy. In view of the pointwise constraint $\abs{u}=1$, one
therefore expects similar results as in the case of minimizing harmonic maps into spheres.

Carbou studied in \cite{carbou} stationary solutions of the Euler-Lagrange equation for $E^\eta$ in the spirit of stationary harmonic maps into spheres. It is shown that these solutions are smooth in the interior of $\Omega$ except for a set of vanishing one-dimensional Hausdorff measure. The proof follows ideas 
taken from the work by Evans \cite{Evans} and uses a monotonicity formula combined with the $\mathcal{H}^1-BMO$ duality. See also Bethuel \cite{Bethuel} or the book by Moser \cite{moser} for stationary harmonic maps into general target manifolds. In \cite{carbou}, regularity results for minimizers of $E^\eta$ are also stated, and it is shown that the singularities of minimizers are isolated in $\Omega$. However, no results concerning boundary regularity have been derived in \nolinebreak \cite{carbou}.

Hardt and Kinderlehrer, using the notion of almost minimizers,
have shown in \cite{hardtkinderlehrer} that minimizers of $E^\eta$ are smooth in the interior of $\Omega$ with the exception of a finite number of 
singularities. Their proof involves modifying partial regularity theory for minimizing harmonic maps as given in the work by Hardt, Kinderlehrer, and Lin \cite{HardtKinderlehrerLin}. Moreover, it is shown in \cite{hardtkinderlehrer} that minimizers are H{\"o}lder continuous near the boundary, provided 
$\Omega$ is a Lipschitz domain that satisfies a certain additional 
assumption which, roughly speaking, ``excludes cusps but not corners'' (see \cite{hardtkinderlehrer} for details). 
It is also shown that the set of singularities 
is completely empty in the interior of $\Omega$ for $\eta>0$ small enough. 
But we remark that a higher regularity result for minimizers up to the 
boundary, namely differentiability, is not stated in 
\cite{hardtkinderlehrer}. Moreover, we want to point out that we have not 
found a reference which proves higher regularity at the boundary for minimizers of $E^\eta$. One would expect such an implication in view of known results for minimizing and stationary harmonic maps into general target manifolds.

In the work by Duzaar and Steffen \cite{DS1} (see also \cite{DS2}), a partial regularity result up to the boundary is proved for mappings $u:\mathcal{M} \to \mathcal{N}$ ($\mathcal{M}$, $\mathcal{N}$ are Riemannian manifolds) which minimize the Dirichlet energy with respect to the free boundary condition $u(\Sigma) \subset \mathcal{S}$. Here, $\Sigma \subset \partial \mathcal{M}$ is a relatively open subset, and $\mathcal{S} \subset \mathcal{N}$ is a submanifold. A generalization to the case of stationary harmonic maps with a free boundary condition is given in the work by Scheven \cite{scheven} (see also \cite{schevendiss}).

The general idea in \cite{DS1} and \cite{scheven} is to use a reflection construction at the boundary in order to establish a situation which is similar to the setting in the interior. In the case of energy minimizers, one can then follow the ideas by Schoen and Uhlenbeck \cite{SchoenUhlenbeck}. To be more precise, a monotonicity formula at the boundary is derived, and a small-energy-regularity theorem is proved by means of a refined version of the \textquotedblleft{}harmonic replacement\textquotedblright{} idea. Moreover, a special coordinate system is introduced in \cite{DS1} and \cite{scheven} in order to obtain differentiability up to the boundary. Here we also want to point out that the methods used in \cite{hardtkinderlehrer} are quite different from the ones used in \cite{DS1} and \cite{scheven}.

In the book by Simon \cite{simon}, a simplified proof of the small-energy-regularity theorem for minimizing harmonic maps $u:\Omega \to \mathcal{N}$ is presented. Here, $\Omega$ is an open subset of the Euclidean space, and $\mathcal{N}$ is a compact Riemannian manifold. The strategy is to apply the monotonicity formula combined with a lemma by Luckhaus to obtain the reverse \Poincare{} inequality for energy minimizing maps. This together with the so-called harmonic approximation lemma enables the author in \cite{simon} to prove the small-energy-regularity theorem by means of the Morrey-Campanato theory.


In this paper we combine the ideas from \cite{hardtkinderlehrer}, \cite{DS1}, \cite{scheven}, and \cite{simon} to prove a higher regularity result for minimizers of $E^\eta$ up to the boundary. More precisely, we introduce in Section \ref{sec:reflection method} the concept of \textquotedblleft{}almost minimizers\textquotedblright{} and use the reflection method to rewrite the minimization problem at the boundary. In Section \ref{sec:monotonicity}, we establish the monotonicity formula (Lemma \ref{thm:monotonicity})  which is used in Section \ref{sec:reversePoincare} to prove the reverse \Poincare{} inequality (Lemma \ref{thm:poincare}). These results and the technique of harmonic approximation are used in Section \nolinebreak \ref{sec:Hoelder regularity} to prove H{\"o}lder regularity under a small-energy assumption (Lemma \ref{thm:hoeldercontinuity}). In Section \nolinebreak \ref{sec:Higher regularity} we improve this regularity result and obtain a small-energy-regularity theorem which is also valid at the boundary (Theorem \ref{thm:linftygradient}). Finally, we use a covering argument to complete the proof of the main theorem.

In the following, we use the short hand notation $\int_\Omega \cdot = \int_\Omega \cdot \,dx$ which always means integration with respect to the Lebesgue measure.
%
%
\section{Almost minimizers and the reflection method}
\label{sec:reflection method}
\paragraph{Euler-Lagrange equation.}
The existence of minimizers for $E^\eta$ follows with the help of the direct method of the calculus of variations and the compact embedding $H^1(\Omega) \hookrightarrow L^2(\Omega)$. Moreover, we can calculate the Euler-Lagrange equation for the energy functional $E^\eta$ as in the case of minimizing harmonic maps into spheres.
\begin{lem}
  \label{lem:eulerlagrange}
  Let $u_\eta \in \hos$ be a minimizer of $E^\eta$. Then we have
  \begin{align*}
    \int_\Omega \hspace{-0,1cm} \nabla u_\eta : \nabla \varphi - \hspace{-0,1cm} \int_\Omega \hspace{-0,1cm}\abs{\nabla u_\eta}^2 \, \, \, u_\eta \cdot \varphi + \eta^2 \hspace{-0,1cm}\int_{\Omega} \hspace{-0,1cm} u_\eta \cdot H[u_\eta]  \, u_\eta \cdot \varphi
    - \eta^2 \hspace{-0,1cm}\int_\Omega \hspace{-0,1cm}H[u_\eta] \cdot \varphi = 0
  \end{align*}
  for all test functions $\varphi \in C_0^1(\setR^3,\setR^3)$.
\end{lem}
\begin{proof}
  For convenience we write $u = u_\eta$ and define for a given $\varphi \in C_0^1(\setR^3,\setR^3)$ 
  the comparison function
  \begin{align*}
    u_t = \frac{u + t \, \varphi}{\abs{u + t \, \varphi}}
  \end{align*}
  for $t \in (-\epsilon,\epsilon)$, where $\epsilon > 0$ is sufficiently small. We remark that $u_t$ belongs to the set of admissible functions $\hos$ and $u_0 = u$. Since $u$ is a minimizer for $E^\eta$, we conclude that $E^\eta (u_0) \le E^\eta(u_t)$
  for all $t \in (-\epsilon,\epsilon)$. This in particular implies that $\frac{\text{d}}{\text{dt}} {E^\eta (u_t)}_{|_{t=0}} = 0$.
  A straightforward calculation as in the case of minimizing harmonic maps into spheres (see for example \cite{simon}) combined with the identity 
  \begin{align*}
  \int_{\setR^3} H[v] \cdot H[w] = - \int_{\Omega} H[v] \cdot w
  \end{align*}
  for all $v,w \in L^2(\Omega,\setR^3) $ gives the result. The lemma is proved.
\end{proof}
\paragraph{Almost minimizers.}
As already announced, we use the notion of almost minimizers from \cite{hardtkinderlehrer} in order to handle the non-local term $H[u]$ in $E^\eta$. Therefore, we define for a given subset $U \subset \setR^3$ the diameter $d(U)$ of $U$ by $d(U) = \sup_{x,y \in U} \abs{x-y}$.
We now show that minimizers of $E^\eta$ are almost minimizers of the Dirichlet energy:
\begin{lem}
  \label{lem:almostminimizer}
  For every $\alpha \in (0,1)$ there exists a constant $C=C(\alpha,\Omega)$ such that
  \begin{align*}
    \int_{\Omega \cap U} \abs{\nabla u_\eta}^2 \le \int_{\Omega \cap U} \abs{\nabla v}^2 + C \, \eta^2 \, d(U)^{1+2 \alpha}
  \end{align*}
  for every minimizer $u_\eta$ of $E^\eta$ and every $v \in \hos$ whenever $v=u_\eta$ on $\Omega \setminus U$ for an open subset $U \subset \subset \setR^3$.
\end{lem}
\begin{proof}
 Let $u_\eta$ be a minimizer of $E^\eta$ and let $v \in \hos$ be such that $v=u_\eta$ on $\Omega \setminus U$, where $U \subset \subset \setR^3$ is on open subset. Since $u_\eta$ is a minimizer, we obtain the estimate
  \begin{align*}
    \int_{\Omega \cap U} \abs{\nabla u_\eta}^2 \le \int_{\Omega \cap U} \abs{\nabla v}^2 + \eta^2 \bigg( \int_{\setR^3} \abs{H[v]}^2 - 
    \int_{\setR^3} \abs{H[u_\eta]}^2 \bigg) \,.
  \end{align*}
  Furthermore, we have that
  \begin{align*}
    \int_{\setR^3} \hspace{-0.1cm}\abs{H[v]}^2 - \int_{\setR^3} \hspace{-0.1cm}\abs{H[u_\eta]}^2 
    = \int_{\setR^3}\hspace{-0.1cm} H[v+u_\eta]\cdot H[v-u_\eta]
    = -\int_{\Omega \cap U} \hspace{-0.1cm} H[v+u_\eta]\cdot (v-u_\eta)
    \le 2 \int_{\Omega \cap U} \hspace{-0.1cm}\abs{H[v+u_\eta]} \, .
  \end{align*}
  Now, we write $p= 3/(1+ 2 \alpha)$ and apply the H{\"o}lder inequality to find
  \begin{align*}
    \int_{\setR^3} \abs{H[v]}^2 - \int_{\setR^3} \abs{H[u_\eta]}^2 &\le C \, d(U)^{1+2 \alpha} \, \norm{H[v+u_\eta]}_{p'}
     \le C(\alpha,\Omega) \, d(U)^{1+2 \alpha} \, ,
  \end{align*}
  where we have used that $H$ is a bounded and linear mapping from $L^q(\Omega,\setR^3)$ to $L^q(\setR^3,\setR^3)$ for every $1<q<\infty$. The latter statement follows directly from the representation formula for $H$ via the Newton-potential (see for example
  Praetorius \cite[Theorem 5.1]{praetorius}). The lemma is proved.
\end{proof}
\paragraph{Choice of coordinates at the boundary.}
In order to obtain a higher regularity result for minimizers up to the boundary, we need a suitable smoothness condition for $\partial \Omega$. Therefore, we assume in the sequel that the domain $\Omega \subset \setR^3$ is a bounded $C^{2,1}$-domain.
\begin{definition}
  \label{def:c21domain}
  A bounded domain $\Omega \subset \setR^3$ is called $C^{2,1}$-domain if for every point 
  $x_0 \in \partial \Omega$ there exist an open neighborhood $U_0$ of $x_0$ in $\setR^3$, 
  an affine transformation $T_0 \xi = R_0 \xi + b_0$ ($R_0 \in SO(3)$, $b_0\in \setR^3$), 
  and a $C^{2,1}$-mapping $g_0:\pi_{\setR^2} \big(T_0(U_0)\big) \to \setR$ such that the following
  equivalences hold true for $x \in T_0(U_0)$:
    \begin{center}
        $x \in T_0(\Omega)$ $\Leftrightarrow$  $x_3 > g_0(x_1,x_2)$ \quad and \quad $x\in T_0(\partial \Omega)$ $\Leftrightarrow$ $x_3 = g_0(x_1,x_2)$.
    \end{center}
   In particular, the boundary of $\Omega$ is locally the graph of a $C^{2,1}$-function, and 
   \nolinebreak $\Omega$ is locally located on one side of the boundary.
\end{definition}
\noindent With the help of the inverse mapping theorem, we now construct parallel coordinates at $\partial \Omega$. These kind of coordinates have also been used in \cite{DS1}, \cite{scheven} and enable us to prove the differentiability of minimizers up to the boundary. In the following we write $B_R = B_R(0)$,
\begin{align*}
 B_R^{+} = \set{(x_1,x_2,x_3) \in B_R \, | \, x_3>0}, \quad \text{and} \quad B_R^{-} = \set{(x_1,x_2,x_3) \in B_R \, | \, x_3<0}
\end{align*}
for every $R>0$. Furthermore, we write $x'=(x_1,x_2,-x_3)$ for every $x=(x_1,x_2,x_3) \in \setR^3$.
\begin{lem}
  \label{lem:coordinates}
 For every point $x_0 \in \partial \Omega$ there exist an open neighborhood $U_0$ of $x_0$ in $\setR^3$ and a 
 $C^1$-diffeomorphism $\psi_0:U_0 \to B_R$ such that the following holds:
 We have $\psi_0 (x_0) = 0$, $\nabla \psi_0$ is Lipschitz continuous, and
    \begin{center}
        $\psi_0 (\Omega \cap U_0)$  $=$  $B_R^{+}$, \quad 
        $\psi_0 (\partial \Omega \cap U_0)$   $=$   $B_R \cap (\setR^2 \times \set{0})$, \quad
        $\psi_0 \big(\overline{\Omega}^{C} \cap U_0\big)$   $=$  $B_R^{-}$.
    \end{center}
 Furthermore, the matrix valued function $\A:B_R \to \setR^{3\times3}$ defined by
 \begin{align*}
   \A(x) 
  =&
  \begin{cases}
   \sqrt{\abs{\det \nabla \psi_0^{-1}(x)}} \, ( \nabla \psi_0^{-1}(x) )^{-1} 
  \quad &\text{if } x_3 \ge 0
    \\
    \sqrt{\abs{\det \nabla \psi_0^{-1}(x')}} \, \diag (1,1,-1) ( \nabla \psi_0^{-1}(x') )^{-1}  \quad &\text{if } x_3 < 0
  \end{cases}
 \end{align*}
  satisfies the following properties:
 \begin{enumerate}
  \item[(i)] $\A\A^T$ is Lipschitz continuous on $B_R$ with Lipschitz constant $L>0$.
  \item [(ii)]$\frac{1}{\beta} \abs{\xi}^2 \le \A(x)\A(x)^T \xi \cdot \xi \le \beta \, \abs{\xi}^2$ 
 for all $\xi \in \setR^3$ and all $x \in B_R$, where $\beta>0$ is a constant independent
 of $\xi$ and $x$.
  \item[(iii)] $M = \sup_{x \in B_R} \abs{\A(x)^{-1}}< \infty$ and $N =  \sup_{x \in B_R} \abs{\A(x)}< \infty$.
  \end{enumerate}
\end{lem}
\begin{proof}
 Let $x_0 \in \partial \Omega$ and let $U_0$, $T_0$, and $g_0$ be as in Definition \ref{def:c21domain}. Without loss of generality, we can assume that $T_0(x_0)=0$, $g_0(0)=0$, and $\nabla g_0(0) = 0$ (apply an affine transformation if necessary). We now define the mapping
  \begin{align*}
    &\varphi_0: \pi_{\setR^2}\big(T_0(U_0)\big) \times \setR \to \setR^3:
    (x_1,x_2,x_3) \mapsto \big(x_1,x_2,g_0(x_1,x_2)\big) + x_3 \big(
    - \partial_1 g_0(x_1,x_2), -\partial_2  g_0(x_1,x_2),1\big) \,.
  \end{align*}
  Obviously, $\varphi_0$ is of class $C^1$, $\varphi_0(0)=0$, and $\nabla \varphi_0 (0) = I$. 
  Thanks to the inverse mapping theorem, we can find an open neighborhood $V_0$ of $0$ in $\setR^3$ and a radius $R>0$ such that $\varphi_0: B_R \to V_0 \subset T_0(U_0)$
  is a $C^1$-diffeomorphism. After choosing a smaller radius $R>0$, we also obtain for $x=(x_1,x_2,x_3) \in B_R$ that
      $\varphi_0 (x_1,x_2,x_3) \in T_0(\Omega)$ if $x_3>0$,
      $\varphi_0 (x_1,x_2,x_3) \in T_0(\partial \Omega)$ if $x_3=0$,
      and 
$\varphi_0 (x_1,x_2,x_3) \in T_0\big(\overline{\Omega}^C\big)$ if $x_3<0$.
  Moreover, one easily verifies that
  \begin{align*}
    \big(\nabla \varphi_0 (x_1,x_2,0)\big)^{-1} = \Gamma(x_1,x_2)
    \left(
    \begin{array}{ccc}
      1+ \big(\partial_2 g_0)^2 & - \partial_1 g_0 \, \partial_2 g_0 & \partial_1 g_0
      \\
      -\partial_1 g_0 \, \partial_2 g_0 & 1 + \big( \partial_1 g_0\big)^2 & \partial_2 g_0
      \\
      - \partial_1 g_0 & - \partial_2 g_0 & 1
    \end{array}
    \right)
  \end{align*}
  for every $(x_1,x_2,0) \in B_R$, where on the right hand side $\partial_1 g_0$, $\partial_2 g_0$ 
  have to be evaluated at $(x_1,x_2)$ and the function $\Gamma(x_1,x_2)$ is given by
    $\Gamma(x_1,x_2)=\big(1+\big(\partial_1 g_0(x_1,x_2) \big)^2 + \big(\partial_2 g_0(x_1,x_2) \big)^2\big)^{-1}$.
 Another elementary calculation shows that
  \begin{align}
  \label{eq:Lipschitz}
    \big(\nabla \varphi_0 (x_1,x_2,0)\big)^{-1} \Big(\big(\nabla \varphi_0 (x_1,x_2,0)\big)^{-1}\Big)^{T} = 
    \left(
    \begin{array}{ccc}
      \ast &  \ast & 0
      \\
      \ast &  \ast & 0
      \\
      0 & 0 & \ast
    \end{array}
    \right) \,,
  \end{align}
  where the matrix entries with \textquotedblleft$\ast$\textquotedblright are not specified in detail.
  We now define
  \begin{align*}
    W_0 = T^{-1}_0(V_0) \qquad \text{and } \qquad \psi_0 = \varphi_0^{-1} \circ T_0 : W_0 \to B_R \,.
  \end{align*}
  Then the function $\psi_0$ satisfies all stated properties after choosing a smaller radius $R>0$. 
  For properties (ii) and (iii) we use standard compactness arguments, and regarding property (i) we remark that 
  \begin{align*}
   \A(x) \A(x)^T = \abs{\det \nabla \varphi_0(x)} \, ( \nabla \varphi_0(x) )^{-1} \big(( \nabla \varphi_0(x) )^{-1}\big)^T
  \quad
  \end{align*}
   for $x_3 \ge 0$ and
  \begin{align*}
   &\A(x) \A(x)^T=  \abs{\det \nabla \varphi_0(x')} \, \diag (1,1,-1) ( \nabla \varphi_0(x') )^{-1} 
    \big(( \nabla \varphi_0(x') )^{-1}\big)^T  \diag (1,1,-1)  
 \end{align*}
 for $x_3<0$. The Lipschitz continuity of $\nabla g_0$ and $\nabla^2 g_0$ implies that $\nabla \varphi_0$ is Lipschitz continuous as well. 
 This combined with \eqref{eq:Lipschitz} guarantees the Lipschitz continuity of 
 $\A\A^T$ on the whole ball $B_R$. The lemma is proved.
\end{proof}

\paragraph{The reflection method.}
In the sequel we fix a point $x_0 \in \partial \Omega$ and choose the corresponding coordinates $\psi_0 : U_0 \to B_R$ from Lemma \ref{lem:coordinates}. For a given minimizer $u_\eta$ of $E^\eta$, we define the function $\overline{u}_\eta \in H^1(B_R,S^2)$ by
\begin{align*}
  \overline{u}_\eta(x) = 
  \begin{cases}
    u_\eta \circ \psi_0^{-1}(x) \quad &\text{if } x_3 \ge 0
    \\
    u_\eta \circ \psi_0^{-1}(x') \quad &\text{if } x_3 < 0
  \end{cases}
\end{align*}
for $x \in B_R$. This means that we first flatten the boundary of $\Omega$ near $x_0$ and then define $\overline{u}_\eta$ on the whole ball $B_R$ by reflection with respect to the hyperplane $\setR^2 \times \set{0}$. We also define
\begin{align*}
  \overline{H}[\overline{u}_\eta](x) =
  \begin{cases}
    H[u_\eta] \circ \psi_0^{-1}(x) \abs{ \det \nabla \psi_0^{-1}(x)} \quad &\text{if } x_3 \ge 0
    \\
    H[u_\eta] \circ \psi_0^{-1}(x') \abs{ \det \nabla \psi_0^{-1}(x')}  \quad &\text{if } x_3 < 0
  \end{cases}
\end{align*}
 for $x \in B_R$. With the help of the transformation formula, we see that $\overline{H}[\overline{u}_\eta]$ belongs to $L^p(B_R,\setR^3)$ for every $1<p<\infty$. In the next two lemmas, we reformulate Lemmas \ref{lem:eulerlagrange} and \ref{lem:almostminimizer} for $\overline{u}_\eta$ on $B_R$.
\begin{lem}
  \label{lem:eulerlagrangeloc}
  The Euler-Lagrange equation for $\overline{u}_\eta$ on $B_R$ reads as
  \begin{align*}
    0 = \int_{B_R} \nabla \overline{u}_\eta \A : \nabla \overline{\varphi} \A - \int_{B_R} \abs{\nabla \overline{u}_\eta \A}^2 \overline{u}_\eta \cdot \overline{\varphi} + \eta^2 \int_{B_R}\overline{u}_\eta \cdot \overline{H}[\overline{u}_\eta]  \, \overline{u}_\eta \cdot \overline{\varphi} - \eta^2 \int_{B_R} \overline{H}[\overline{u}_\eta] \cdot \overline{\varphi}
  \end{align*}
  for all test functions $\overline{\varphi} \in C_0^\infty(B_R,\setR^3)$, where the matrix valued function $\A$ is defined as in Lemma \ref{lem:coordinates}.
\end{lem}
\begin{proof}
 For a given $\overline{\varphi} \in C_0^\infty(B_R,\setR^3)$, we define $\varphi = \overline{\varphi} \circ \psi_0 \in C^1_0(\setR^3,\setR^3)$ with $\support \varphi \subset U_0$. Lemma \ref{lem:eulerlagrange} implies the identity
  \begin{align*}
    0=\int_{\Omega \cap U_0} \nabla u_\eta : \nabla \varphi - \int_{\Omega \cap U_0} \abs{\nabla u_\eta}^2 \, \, \, u_\eta \cdot \varphi 
 + \eta^2 \int_{{\Omega \cap U_0}}  u_\eta \cdot H[u_\eta] \, u_\eta \cdot \varphi - \eta^2 \int_{\Omega \cap U_0} H[u_\eta] \cdot \varphi \,.
  \end{align*}
  From here we obtain that
  \begin{align*}
    0 =&\int_{B_R^{+}} \nabla \overline{u}_\eta \A : \nabla \overline{\varphi} \A - \int_{B_R^{+}} \abs{\nabla \overline{u}_\eta \A}^2 \overline{u}_\eta \cdot \overline{\varphi} + \eta^2 \int_{B_R^{+}} \overline{u}_\eta \cdot \overline{H}[\overline{u}_\eta] \, \, \overline{u}_\eta \cdot \overline{\varphi} - \eta^2 \int_{B_R^{+}} \overline{H}[\overline{u}_\eta] \cdot \overline{\varphi}\, ,
  \end{align*}
   where we have used the transformation formula and $\psi_0(\Omega \cap U_0) = B_R^{+}$.
  Instead of $\overline{\varphi}$, we can also use the test function $\overline{\overline{\varphi}} \in C_0^\infty(B_R,\setR^3)$ defined by
   $\overline{\overline{\varphi}}(x) = \overline{\varphi}(x')$
for $x \in B_R$.
Another application of the transformation formula leads to the same identity for $\overline{u}_\eta$ and $\overline{\varphi}$, but this time with integrals taken over $B_R^{-}$. The lemma is proved.
\end{proof}
Similarly, we find the following lemma:
\begin{lem}
  \label{lem:almostminimizerloc}
  For every $\alpha \in (0,1)$ there is a constant $C=C(\alpha,\Omega)$ such that
  \begin{align*}
    \int_{U} \abs{\nabla \overline{u}_\eta \A}^2 \le \int_{U} \abs{\nabla \overline{v} \A}^2 + C \, \eta^2 \, d(U)^{1+2 \alpha}
  \end{align*}
  for every function $\overline{v} \in \hbs$ whenever $\overline{v}=\overline{u}_\eta$ on $B_R \setminus U$ for an open subset $U \subset \subset B_R$.
\end{lem}
\begin{proof}
 Let $\overline{v} \in \hbs$ be such that $\overline{v}=\overline{u}_\eta$ on 
 $B_R \setminus U$ for an open subset $U \subset \subset B_R$. We set $V = \psi_0^{-1} (U) \subset \subset U_0$ and choose a function $\varphi \in C_0^\infty(U_0)$ with  $\varphi \equiv 1$ on $V$. Since $\Omega$ is a bounded $C^{2,1}$-domain, we can extend $u_\eta$ to a function in $H^1(\setR^3,\setR^3)$. We now define the functions $v = \overline{v} \circ \psi_0 \in H^1(U_0,S^2)$ and
 $w = \varphi v + (1- \varphi) u_\eta \in H^1(\setR^3,\setR^3)$. From this definitions we read off that $w \in H^1(\Omega,S^2)$, $w=v$ on $\Omega \cap V$, and $w = u_\eta$ on $\Omega \setminus V$. Therefore, we obtain from Lemma \ref{lem:almostminimizer} the estimate
\begin{align*}
 \int_{\Omega \cap V} \abs{\nabla u_\eta}^2 \le \int_{\Omega \cap V} \abs{\nabla v}^2 + C \, \eta^2 d(V)^{1+2 \alpha} \,.
\end{align*}
An application of the transformation formula shows that
\begin{align*}
 \int_{\Omega \cap V} \abs{\nabla u_\eta}^2 = \int_{B_R^+ \cap U} \abs{\nabla u_\eta}^2 \circ \psi_0^{-1} \abs{\det \nabla \psi_0^{-1}} = \int_{B_R^+ \cap U} \abs{\nabla \overline{u}_\eta \A}^2 \,,
\end{align*}
where we have used the identity $\nabla (u_\eta \circ \psi_0^{-1}) = \nabla u_\eta \circ \psi_0^{-1} \, \nabla \psi_0^{-1}$. In the same manner we obtain
\begin{align*}
 \int_{\Omega \cap V} \abs{\nabla v}^2 = \int_{B_R^+ \cap U} \abs{\nabla \overline{v} \A}^2 \,.
\end{align*}
Thanks to the Lipschitz continuity of $\psi_0^{-1}$ with Lipschitz constant C, we have the estimate $d(V) \le C d(U)$, hence
\begin{align*}
 \int_{B_R^+ \cap U} \abs{\nabla \overline{u}_\eta \A}^2 \le
 \int_{B_R^+ \cap U} \abs{\nabla \overline{v} \A}^2 + C \, \eta^2 d(U)^{1+2\alpha} \,.
\end{align*}
Instead of $\overline{v}$, we can also use $\overline{\overline{v}} \in \hbs$ defined by
 $\overline{\overline{v}}(x_1,x_2,x_3) = \overline{v}(x_1,x_2,-x_3)$.
In particular, we have $\overline{\overline{v}} = \overline{u}_\eta$ on $B_R \setminus \tilde{U}$, where $\tilde{U} = \set{(x_1,x_2,x_3) \in \setR^3 \, | \,(x_1,x_2,-x_3) \in U}$.
This shows
\begin{align*}
 \int_{B_R^+ \cap \tilde{U}} \abs{\nabla \overline{u}_\eta \A}^2 \le
 \int_{B_R^+ \cap \tilde{U}} \abs{\nabla \overline{\overline{v}} \A}^2 + C \, \eta^2 d(\tilde{U})^{1+2\alpha} \,.
\end{align*}
The transformation formula and $d(U) = d(\tilde{U})$ imply the estimate
\begin{align*}
 \int_{B_R^- \cap U} \abs{\nabla \overline{u}_\eta \A}^2 \le
 \int_{B_R^- \cap U} \abs{\nabla \overline{v} \A}^2 + C \, \eta^2 d(U)^{1+2\alpha} \,.
\end{align*}
The lemma is proved.
\end{proof}
In view of Lemmas \ref{lem:eulerlagrange}, \ref{lem:almostminimizer}, \ref{lem:coordinates}, \ref{lem:eulerlagrangeloc}, and  \ref{lem:almostminimizerloc}, we formulate the following assumptions:
\begin{assumption}[$\mathbf{A1}$]
 The set $\Omega \subset \setR^3$ is a bounded domain, and the matrix valued function $\A:\Omega \to \setR^{3\times 3}$ meets the following conditions:
 \begin{enumerate}
  \item[(i)] $\A\A^T$ is Lipschitz continuous on $\Omega$ with Lipschitz constant $L>0$.
  \item[(ii)] $\frac{1}{\beta} \abs{\xi}^2 \le \A(x)\A(x)^T \xi \cdot \xi \le \beta \, \abs{\xi}^2$ 
 for all $\xi \in \setR^3$ and all $x \in \Omega$, where $\beta>0$ is a constant independent
 of $\xi$ and $x$.
  \item[(iii)] $M = \sup_{x \in \Omega} \abs{\A(x)^{-1}}< \infty$ and $N =  \sup_{x \in \Omega} \abs{\A(x)}< \infty$.
  \end{enumerate}
 Furthermore, the function $u \in H^1(\Omega,S^2)$ satisfies the estimate
\begin{align*}
  \int_{U} \abs{\nabla u \A}^2 \le \int_{U} \abs{\nabla v \A}^2 + C \, \eta^2 \, d(U)^{1+2 \alpha}
\end{align*}
 for every $\alpha \in (0,1)$ and every comparison function $v \in \hos$ whenever $v=u$ on $\Omega \setminus U$ for an open subset $U \subset \subset \Omega$. Here, the constant $\eta>0$ is fixed, and $C=C(\alpha)>0$ depends only on $\alpha$.
\end{assumption}
\begin{assumption}[$\mathbf{A2}$]
 Let $u$ and $\A$ be as in (A1). There is a function  $f \in \bigcap_{1<p<\infty} L^p(\Omega,\setR^3)$
such that
 \begin{align*}
    0 =& \int_{\Omega} \nabla u \A : \nabla \varphi \A - \int_{\Omega} \abs{\nabla u \A}^2 u \cdot \varphi -\eta^2 \int_\Omega f \cdot \varphi
  \end{align*}
 for every $\varphi \in C_0^\infty(\Omega,\setR^3)$.
\end{assumption}
\begin{remark}
 In the sequel we say that a constant $C$ depends on the ``typical parameters'' if \nolinebreak $C$ depends only on $L,M,N,\beta,d(\Omega)$, and certain $L^p$-norms of $f$.
\end{remark}

\section{The monotonicity formula}
\label{sec:monotonicity}
In this section we prove the monotonicity formula for functions which satisfy assumption (A1). First, we derive a monotonicity formula for points $a \in \Omega$ with $\A(a)\A(a)^T = I$, and afterwards, we use a coordinate transformation in order to obtain the result for arbitrary points.
\begin{lem}
\label{lem:monotonicity}
 Suppose $u$ and $\A$ satisfy (A1) and $a \in \Omega$ is a point such that $\A(a)\A(a)^T = I$.
 Then the mapping
  \begin{align*}
    t \mapsto e^{c \, t} \, \frac{1}{t} \int_{B_t(a)} \, 
    \abs{\nabla u \A}^2 + C \, \eta^2 \, t^{2 \alpha}
  \end{align*}
  is monotone increasing for every $\alpha \in (0,1)$, where
  the constants $c,C>0$ depend only on $\alpha,L,M$, and $d(\Omega )$.
\end{lem}
\begin{proof}
  As in the case of minimizing harmonic maps into spheres, we define the comparison function
  \begin{align*}
    u_t(x) = 
    \begin{cases}
      u(x) & \text{if }x \in \Omega \setminus \overline{B_t(a)} 
      \\
      u(a+t \frac{x-a}{\abs{x-a}}) & \text{if }x \in B_t(a)
    \end{cases} 
  \end{align*}
  for $x \in \Omega$ and $t>0$. In particular, 
  we have  $u_t = u$ on $\Omega \setminus 
  \overline{B_t(a)}$ and $u_t \in \hos$. Moreover, we find for the derivative  $\nabla u_t$ the identity
  \begin{align*}
    \nabla u_t(x) = \nabla u \left (a+t \frac{x-a}{\abs{x-a}}\right) \left(\frac{t}
    {\abs{x-a}} I - t \frac{(x-a) \otimes (x-a)}{\abs{x-a}^3}\right)
  \end{align*}
  for every $x \in B_t(a)$. Thus
  \begin{align*}
     \int_{B_t(a)} \abs{\nabla u_t \A}^2 
    = & \int_0^t \int_{\partial B_t(a)} \biggabs{\nabla u(x) (I - \nu(x) \otimes \nu(x) ) \A\left( \frac{r}{t} (x-a) +a \right)}^2 \, do(x) \, dr \, , 
  \end{align*}
  where $\nu(x) = (x-a)/\abs{x-a}$ is the unit outer normal of $\partial B_t(a)$ at $x \in \partial B_t(a)$. In the sequel we use the abbreviation
    $\B= \B(x,r) = \A\left( \frac{r}{t} (x-a) +a \right)$
  for $x \in \partial B_t(a)$ and $0 \le r \le t$. Furthermore, we use the identity $A:CD = AD^T:C$
  for all $A,B,C \in \setR^{3\times3}$ to find
  \begin{align*}
   \abs{\nabla u (I- \nu \otimes \nu) \B}^2 
 =& \nabla u \B \B^T : \nabla u - 2 \nabla u \B \B^T : \nabla u \nu \otimes \nu + \nabla u \nu \otimes \nu \B \B^T : \nabla u \nu \otimes \nu \,.
  \end{align*}
  By introducing $\A\A^T= \A(x)\A(x)^T$ and $I$, we obtain the following decomposition:
  \begin{align*}
   \abs{\nabla u (I- \nu \otimes \nu) \B}^2 
  =& \abs{\nabla u \A}^2 + \nabla u (\B\B^T - \A\A^T) : \nabla u - 2 \nabla u (\B\B^T -I): \nabla u \nu \otimes \nu 
  \\
  &- 2 \nabla u : \nabla u \nu \otimes \nu + \nabla u \nu \otimes \nu (\B\B^T - I) : \nabla u \nu \otimes \nu + \abs{\nabla u \nu \otimes \nu}^2 \,.
  \end{align*}
  We have $M=\sup_{x \in \Omega} \abs{\A(x)^{-1}}<\infty$,
  $\abs{\B\B^T - I} \le L \, r \,$ , and
  $\abs{\B\B^T - \A\A^T} \le L \, (t-r)$
  on $\partial B_t(a)$, thanks to the properties of $\A$.
  Moreover, a simple calculation shows that
  \begin{align*}
   \abs{\nabla u \nu \otimes \nu}^2 - 2 \nabla u : \nabla u \nu \otimes \nu \le 0 \,.
  \end{align*}
  A combination of these facts yields the following estimate:
  \begin{align*}
   \abs{\nabla u (I- \nu \otimes \nu) \B}^2  &\le \abs{\nabla u \A}^2 
   + L t\abs{\nabla u}^2  + 2Lr \abs{\nabla u}^2 \le \abs{\nabla u \A}^2 + L M^2 t\abs{\nabla u \A}^2  + 2 L M^2 r\abs{\nabla u \A}^2  \,.
  \end{align*}
  In particular, we obtain
  \begin{align*}
    & \int_{B_t(a)} \abs{\nabla u_t \A}^2 
   \le ( 1 + 2 \, L \, M^2 \, t ) \, t
   \int_{\partial B_t(a)} \abs{ \nabla u \A}^2 \, do(x)\,.
  \end{align*}
  Since $u_t$ is an admissible comparison function, we obtain from (A1) the inequality
  \begin{align*}
    \int_{B_t(a)} \abs{\nabla u \A}^2 \le ( 1+ 2 \, L \, M^2 \, t ) \, t \int_{\partial B_t(a)} \abs{\nabla u \A}^2 \, do(x)+ C \, \eta^2 \, t^{1+2 \alpha} \, ,
  \end{align*}
  where $C=C(\alpha)>0$. Rearrangement leads to
  \begin{align*}
    0 \le& ( 1+ 2 \, L \, M^2 \, t ) \bigg( t \, \int_{\partial B_t(a)} \abs{\nabla u \A}^2 \, do(x)- \int_{B_t(a)} \abs{\nabla u \A}^2
    \\
    &\qquad \qquad \qquad \qquad+ \frac{ 2 \, L \, M^2 t}{ 1 + 2 \, L \, M^2 \, t}  \int_{B_t(a)} \abs{\nabla u \A}^2 + \frac{C}
     { 1+ 2 \, L \, M^2 \, t } \, \eta^2 \, t^{1 + 2 \alpha} \bigg) \,.
  \end{align*}
  In particular, we have that
  \begin{align*}
    0 \le & \, t \, \int_{\partial B_t(a)} \abs{\nabla u \A}^2 \, do(x)- \int_{B_t(a)} \abs{\nabla u \A}^2 
    + 2 \, L \, M^2 \, t 
    \int_{B_t(a)} \abs{\nabla u \A}^2 + C \eta^2 \, t^{1 + 2 \alpha} \,.
  \end{align*}
  We now introduce the constant $c = 2 \, L \, M^2$ and multiply by $t^{-2} \, e^{c \, t}$ in order to obtain the following estimate:
  \begin{align*}
    0 \le \, \frac{\text{d}}{\text{dt}} \bigg( e^{c \, t} \frac{1}{t} \int_{B_t(a)} \abs{\nabla u \A}^2 
      + C e^{c \, d(\Omega )} \, \eta^2 \, t^{2 \alpha} \bigg) \,.
  \end{align*}
  The lemma is proved.
\end{proof}
With the help of an affine coordinate transformation, we obtain the monotonicity formula for arbitrary points. Since the proof is rather straightforward, we omit the proof here.
\begin{lem}
 \label{thm:monotonicity}
 Suppose $u$ and $\A$ satisfy (A1). Then for every $a \in \Omega$ there is an invertible affine transformation $T_a: \setR^3 \to \setR^3$ such that
 \begin{align*}
    t \mapsto e^{c \, t} \, \frac{1}{t} \int_{T_a(B_t(a))} \abs{\nabla u \A}^2 + C \, \eta^2 \, t^{2 \alpha}
  \end{align*}
 is monotone increasing for every $\alpha \in (0,1)$, where the positive constants $c$ and $C$ depend on $\alpha$, but are independent of the point $a$. Furthermore, $T_a$ satisfies $T_a(a)=a$,
 \begin{gather*}
  c \le \abs{\nabla T_a}, \, \abs{(\nabla T_a)^{-1}} \le C , \quad \text{and} \quad
  c \le \abs{\det \nabla T_a} \le C \, ,
 \end{gather*}
 where the positive constants $c$ and $C$ are independent of the point $a$.
\end{lem}

\section{The reverse \Poincare{} inequality}
\label{sec:reversePoincare}
In this section we apply the monotonicity formula in order to prove the reverse \Poincare{} inequality, which marks an important step towards regularity. Before we do so, we need an additional tool that can be used to construct suitable comparison maps. More precisely, we utilize a general lemma by Luckhaus (see for example \cite{Luckhaus1} or \cite{Luckhaus2}) in the version stated below. But first of all, we introduce the abbreviations
\begin{align*}
 \ui_B  = \dashint\nolimits_{\hspace{-0.16cm}B} u = \frac{1}{\abs{B}} \int_B u
\end{align*}
for the average of a given function $u$ on a given set $B$ whenever this is well-defined. If $B=B_\rho(y)$ is a ball, we simply write $\ui_{y,\rho} = \ui_{B_\rho(y)}$. In the sequel we also make use of the inequality
\begin{align}
 \label{eq:mean value}
 \int_B \abs{u - \ui_B}^2 \le \int_B \abs{u-\lambda}^2 \, ,
\end{align}
where $u \in L^1(B,\setR^3)$ and $\lambda \in \setR^3$.
\begin{lem}[Luckhaus]
 \label{lem:Luckhaus}
 For every $\Lambda>0$ there exist positive constants $\delta_0=\delta_0(\Lambda)$ and $C=C(\Lambda)$ such that the following holds: For every $\epsilon \in (0,1)$ and every $u \in H^1(B_\rho(y),S^2)$ with 
\begin{align*}
 \rho^{-1} \int_{B_\rho(y)} \abs{\nabla u}^2 \le \Lambda \qquad \text{and} \qquad
 \epsilon^{-6} \rho^{-3} \int_{B_\rho(y)} \abs{u - \ui_{y,\rho}}^2 \le \delta_0^2
\end{align*}
 there is a $\sigma \in (3/4 \,\rho,\rho)$ and a function $w \in H^1(B_\rho(y),S^2)$ such that $w = u$ on \nolinebreak $\partial B_\sigma(y)$ and
\begin{align*}
 \sigma^{-1} \int_{B_\sigma(y)} \abs{\nabla w}^2 \le \epsilon \rho^{-1} \int_{B_\rho(y)} \abs{\nabla u}^2 + C\epsilon^{-1}  \rho^{-3} \int_{B_\rho(y)} \abs{u - \ui_{y,\rho}}^2 \,.
\end{align*}
\end{lem}
\begin{proof}
 The version at hand is taken from \cite[Section 2.7, Corollary 1]{simon}.
\end{proof}
Under the assumptions of Lemma \ref{thm:monotonicity}, there exist constants $0<c_1<c_2$ such that
\begin{align}
 \label{eq:inclusion}
 B_{c_1 r}(a) \subset T_a(B_r(a)) \subset B_{c_2 r} (a)
\end{align}
for every $a \in \Omega$ and every $r>0$. Furthermore, we define $c_0 = c_1/c_2 \in (0,1)$.
We can now prove the following version of the reverse \Poincare{} inequality:
\begin{lem}
 \label{thm:poincare}
 Suppose $u$ and $\A$ satisfy (A1). Then for every $\alpha \in (0,1)$ and $\Lambda > 0$ there exists a constant $C>0$ depending only on $\alpha$, $\Lambda$, and the typical parameters with the following property: If $B_{2r}(x_0) \subset \Omega$ is a ball such that $r^{-1} \int_{B_{2r}(x_0)} \abs{\nabla u \A}^2 \le \Lambda$, then we have the estimate
 \begin{align*}
    \frac{1}{r} \int_{B_{c_0 \, \frac{r}{2}}(x_0)} \abs{\nabla u}^2 \le C \, \bigg( \frac{1}{r^3} \int_{B_r(x_0)} \abs{u- \ui_{x_0,r}}^2 
  + \eta^2 \, r^{2 \alpha} \bigg) \,.
  \end{align*}
\end{lem}
\begin{proof}
  Let $\alpha \in (0,1)$, $\Lambda>0$ be given and let $B_{2r}(x_0) \subset \Omega$ be such that $r^{-1} \int_{B_{2r}(x_0)} \abs{\nabla u \A}^2 \le \Lambda$. For an arbitrary ball $B_s(y_0) \subset B_{c_0 r}(x_0)$, we can estimate as follows
  \begin{align*}
    \frac{1}{s} \int_{B_s(y_0)} \abs{\nabla u}^2 
    \le  \frac{\beta}{s} \int_{T_{y_0}\big(B_\frac{s}{c_1}(y_0)\big)}  \abs{\nabla u \A}^2
    \le \frac{\beta}{c_1} \bigg( \,e^{c \, \frac{s}{c_1}} \Big( \frac{s}{c_1} \Big)^{-1} 
    \int_{T_{y_0}\big(B_\frac{s}{c_1}(y_0)\big)}  \abs{\nabla u \A}^2 + C \, \eta^2 \, \Big( \frac{s}{c_1} \Big)^{2 \alpha} \,\bigg) \,,
  \end{align*}
  where we have used (A1) and \eqref{eq:inclusion}, and $T_{y_0},c,C$ are taken from  Lemma \ref{thm:monotonicity}. 
 Since $s/c_1\le r/c_2$ and the inclusion $B_r(y_0) \subset B_{2r}(x_0)$ holds, we obtain with the help of the monotonicity formula the estimate
  \begin{align}
  \label{eq:reverse Poincare 1}
    \frac{1}{s} \int_{B_s(y_0)} \abs{\nabla u}^2
    \le  \frac{C}{r} \int_{B_{2r}(x_0)} \abs{\nabla u \A}^2 + C \, \eta^2 \, r^{2 \alpha}
    \le& \tilde{\Lambda} \,,
  \end{align}
  where $\tilde{\Lambda}$ is independent of $x_0$, $y_0$, $r$, and $s$.
  \\
  Claim: For every $\epsilon > 0$ there is a $\delta_0 = \delta_0(\epsilon, \tilde{\Lambda})>0$ with the following property: If $B_s(y_0) \subset B_{c_0r}(x_0)$ is a ball such that
  \begin{align*}
    \frac{1}{s^3} \int_{B_s(y_0)} \abs{u- \ui_{y_0,s}}^2 \le \delta_0^2 \, ,
  \end{align*}
  then we have the estimate
  \begin{align*}
    \frac{1}{s} \int_{B_{c_0 \frac{s}{2}} (y_0)} \abs{\nabla u}^2 \le \frac{C \epsilon}{s} \int_{B_s(y_0)} \abs{\nabla u}^2 + 
       \frac{C}{\epsilon \, s^3} \int_{B_s(y_0)} \abs{u - \ui_{y_0,s}}^2 + C \, \eta^2 \, s^{2 \alpha} \, ,
  \end{align*}
  where $C=C(\alpha, \tilde{\Lambda})$ is a constant.
  \\
  Proof of Claim:
  Let $\epsilon >0$ be given and let $\delta_1 = \delta_1 (\tilde{\Lambda})$ be the constant from the Luckhaus lemma  with respect to $\tilde{\Lambda}$ (see Lemma \ref{lem:Luckhaus}). We now define \nolinebreak  $\delta_0$ by 
  $\delta_0^2 = \epsilon^6 \, \delta_1^2$ and obtain for an arbitrary ball $B_s(y_0) \subset B_{c_0r}(x_0)$ with $s^{-3} \int_{B_s(y_0)} \abs{u- \ui_{y_0,s}}^2 \le \delta_0^2$ the estimate
  \begin{align*}
    \epsilon^{-6} \, s^{-3} \int_{B_s(y_0)} \abs{u- \ui_{y_0,s}}^2 \le \epsilon^{-6} \, \delta_0^2 = \delta_1^2 \,.
  \end{align*}
  This combined with the Luckhaus lemma and \eqref{eq:reverse Poincare 1} implies the existence of a $\sigma \in \big(3/4 \,s,s\big)$ and a function $w \in H^1(B_s(y_0),S^2)$ with $w=u$ on $\partial B_\sigma (y_0)$ such that
  \begin{align}
   \label{eq:reverse Poincare 2}
    \frac{1}{\sigma} \int_{B_{\sigma}(y_0)} \abs{\nabla w}^2 \le \frac{\epsilon}{s} \int_{B_s(y_0)} \abs{\nabla u}^2
    + \frac{C}{\epsilon \, s^3} \int_{B_s(y_0)} \abs{u- \ui_{y_0,s}}^2 \, ,
  \end{align}
  where $C=C(\tilde{\Lambda})$ is a constant. Moreover, we obtain by (A1) and \eqref{eq:inclusion} the estimate
  \begin{align*}
     \frac{1}{s} \hspace{-0,1cm}\int_{B_{c_0 \frac{s}{2}} (y_0)} \hspace{-0,1cm}\abs{\nabla u}^2
    \le  C  \bigg( e^{c \frac{s}{2  c_2}} \Big(\frac{s}{2  c_2}\Big)^{-1} \hspace{-0,2cm}
    \int_{T_{y_0}\big(B_{\frac{s}{2  c_2}}(y_0)\big)} \hspace{-0,1cm}\abs{\nabla u \A}^2  + C \eta^2\Big(\frac{s}{2  c_2}\Big)^{2 \alpha} 
    \bigg) .
  \end{align*}
  With the help of the monotonicity formula (see Lemma \ref{thm:monotonicity}), \eqref{eq:inclusion}, and the fact that $s/2 \le \sigma$, we get the estimate
  \begin{align*}
     \frac{1}{s} \hspace{-0,1cm}\int_{B_{c_0 \frac{s}{2}} (y_0)} \abs{\nabla u}^2
    \le  \frac{C}{\sigma} \int_{B_{\sigma}(y_0)} \hspace{-0,1cm}\abs{\nabla u \A}^2 + C  \eta^2  \sigma^{2 \alpha} \,.
  \end{align*}
 In view of (A1), we define the comparison map $v \in H^1(\Omega,S^2)$ by
 \begin{align*}
  v(x) =
  \begin{cases}
   u(x) & \qquad \text{if } x \in \Omega\setminus B_\sigma(y_0)
   \\
   w(x) & \qquad  \text{if } x \in B_\sigma(y_0)
  \end{cases}
 \end{align*}
 and obtain together with \eqref{eq:reverse Poincare 2} that
  \begin{align*}
     \frac{1}{s} \int_{B_{c_0 \frac{s}{2}} (y_0)} \abs{\nabla u}^2
    \le  \frac{C \epsilon}{s} \int_{B_s(y_0)} \abs{\nabla u}^2 + \frac{C}{\epsilon \,s^3} 
    \int_{B_s(y_0)} \abs{u - \ui_{y_0,s}}^2 + C \eta^2 s^{2 \alpha} \, .
  \end{align*}
  The claim is proved.
  \\
  If now $B_s(y_0) \subset B_{c_0r}(x_0)$ is a ball such that
  \begin{align*}
    \frac{1}{s^3} \int_{B_s(y_0)} \abs{u- \ui_{y_0,s}}^2 > \delta_0^2 \, ,
  \end{align*}
  then we obtain together with \eqref{eq:reverse Poincare 1} the estimate
  \begin{align*}
    \frac{1}{s^3} \int_{B_s(y_0)} \abs{u-\ui_{y_0,s}}^2 \ge \delta_0^2 = \delta_0^2 \frac{\tilde{\Lambda}}{\tilde{\Lambda}} \ge 
    \frac{\delta_0^2}{\tilde{\Lambda}} \frac{1}{s} \int_{B_{c_0 \frac{s}{2}}(y_0)} \abs{\nabla u}^2 \,.
  \end{align*}
  A combination of the above estimates shows that
  \begin{align}
   \label{eq:reverse Poincare 3} \hspace{-0,1cm}
    \frac{1}{s} \hspace{-0,1cm}\int_{B_{c_0 \frac{s}{2}}(y_0)} \hspace{-0,1cm}\abs{\nabla u}^2 \le  \frac{C\epsilon}{s} \hspace{-0,1cm}\int_{B_s(y_0)} \hspace{-0,1cm}\abs{\nabla u}^2
    + \frac{C(\epsilon)}{s^3} \hspace{-0,1cm}\int_{B_s(y_0)} \hspace{-0,1cm}\abs{u- \ui_{y_0,s}}^2 + C\eta^2s^{2 \alpha}
  \end{align}
  for all balls $B_s(y_0) \subset B_{c_0 r}(x_0)$. We still have to choose $\epsilon >0$ and do this in the following way:
First of all, we define
  \begin{align*}
    \mathcal{A}= \set{B_\sigma \, | \, B_{2  \sigma} \subset B_{c_0 r}(x_0)} \quad \text{and} \quad \mathcal{Q} = 
    \sup_{B_\sigma \in \mathcal{A}} \sigma^2 \int_{B_\sigma} \abs{\nabla u}^2 \, .
  \end{align*}
  Let now $B_\sigma(x) \in \mathcal{A}$ be given. There are $N$ points $x_1, \dots, x_N \in B_{\sigma}(x)$ such that
  \begin{align*}
    B_\sigma(x) \subset \bigcup_{i=1}^N B_{c_0 \frac{\sigma}{4}}(x_i) \, ,
  \end{align*}
  where $N$ is independent of the special choice of $B_\sigma(x) \in \mathcal{A}$. We remark that 
  $B_\sigma(x_i) \subset B_{2 \sigma} (x) \subset B_{c_0 r}(x_0)$ for every $i \in \set{1,\dots,N}$ and obtain together with \eqref{eq:mean value} and \eqref{eq:reverse Poincare 3} that
  \begin{align*}\hspace{-0,1cm}
    \sigma^2 \hspace{-0,2cm}\int_{B_\sigma(x)} \hspace{-0,1cm}\abs{\nabla u}^2
    \le & \sigma^2 \sum_{i=1}^N \int_{B_{c_0 \frac{\sigma}{4}}(x_i)} \hspace{-0,1cm}\abs{\nabla u}^2
    \\
    \le & \sum_{i=1}^N \hspace{-0,1cm}\bigg( \hspace{-0,1cm}C \epsilon \sigma^2 \hspace{-0,2cm}\int_{B_{\frac{\sigma}{2}}(x_i)} \hspace{-0,1cm}\abs{\nabla u}^2 + C(\epsilon)
    \hspace{-0,2cm}\int_{B_\frac{\sigma}{2}(x_i)} \hspace{-0,1cm}\abs{u- \ui_{x_i,\frac{\sigma}{2}}}^2 + C \eta^2 \sigma^{3+2 \alpha} \hspace{-0,1cm}\bigg)
    \\
    \le & C  \epsilon \mathcal{Q} + C(\epsilon) \hspace{-0,2cm}\int_{B_r(x_0)}\hspace{-0,1cm} \abs{u-\ui_{x_0,r}}^2 + C  \eta^2  r^{3+2\alpha} \,.
  \end{align*}
  Since $B_\sigma(x) \in \mathcal{A}$ was arbitrary, we conclude
  \begin{align*}
    \mathcal{Q} \le C  \epsilon  \mathcal{Q} + C(\epsilon) \int_{B_r(x_0)} \abs{u-\ui_{x_0,r}}^2 
    + C  \eta^2  r^{3+2\alpha} \, .
  \end{align*}
  Now we choose $\epsilon = 1/(2 C)$ in order to find
  \begin{align*}
    \mathcal{Q} \le C \bigg( \int_{B_r(x_0)} \abs{u-\ui_{x_0,r}}^2  +  \eta^2 r^{3+2\alpha} \bigg) \,.
  \end{align*}
  The lemma is proved since $B_{c_0 \frac{r}{2}}(x_0) \in \mathcal{A}$.
\end{proof}
 With the help of the monotonicity formula, we easily obtain the following corollary:
\begin{corollary}
  \label{cor:poincare}
  Suppose $u$ and $\A$ satisfy assumption (A1). Then for every $\alpha \in (0,1)$ and $\Lambda >0$ 
  there is a constant $C>0$ depending only on $\alpha$, $\Lambda$, and the typical parameters with the following property:
  If $B_{3r}(x_0) \subset \Omega$ is a ball such that $\frac{1}{r} \int_{B_{3r}(x_0)} \abs{\nabla u \A}^2 \le \Lambda$, then
  \begin{align*}
    \frac{1}{s} \int_{B_{c_0 \frac{s}{2}}(y)} \abs{\nabla u}^2 \le C \, \bigg(\frac{1}{s^3} \int_{B_s(y)} \abs{u- \ui_{y,s}}^2 + \eta^2  
    s^{2 \alpha} \bigg)
  \end{align*}
  for all $y \in B_r(x_0)$ and $0< s \le c_0 r$.
\end{corollary}

\section{H{\"o}lder regularity}
\label{sec:Hoelder regularity}
In this section we prove the H{\"o}lder regularity for functions $u$ which satisfy assumptions (A1) and (A2). For this we make use of the well-known Campanato lemma (see \cite{giaquinta} or \cite{simon} for a proof):
\begin{lem}[Campanato]
\label{lem:Campanato}
 Let $v \in L^2(B_{2R}(x_0))$ with $B_{2R}(x_0) \subset \setR^3$, $\alpha \in (0,1)$, and $\beta>0$ be such that
 \begin{align*}
  \rho^{-3} \int_{B_\rho(y)} \abs{v - \overline{v}_{y,\rho}}^2 \le \beta^2 \Big(\frac{\rho}{R}\Big)^{2 \alpha}
 \end{align*}
 for all $y \in B_R(x_0)$ and all $0<\rho\le R$. Then we have $v \in C^{0,\alpha}(B_{R}(x_0))$ and
 \begin{align*}
  \abs{v(x)-v(y)} \le C \beta \bigg(\frac{\abs{x-y}}{R}\bigg)^{\alpha}
 \end{align*}
 for all $x,y \in B_R(x_0)$, where $C$ depends only on $\alpha$.
\end{lem}
In order to obtain the necessary decay estimate, we apply in the following the technique of harmonic approximation. This technique goes back to Simon and was, for example, used to simplify the original proof of the small-energy-regularity theorem for minimizing harmonic maps (see \cite{simon}). The idea is to compare functions which are \textquotedblleft{}approximately harmonic\textquotedblright{} with harmonic functions. A generalization to elliptic bilinear forms is given in the paper by Duzaar and Grotowski \cite{duzaargrotowski} as stated in the special version below. For a bilinear mapping 
$\mathcal{A}:\setR^{3\times3} \times \setR^{3\times3} \to \setR$, we call a function $w \in H^1(B_\rho(x_0),\setR^3)$ with $B_\rho(x_0) \subset \setR^3$ \textquotedblleft{}$\mathcal{A}$-harmonic\textquotedblright{} if it satisfies
\begin{align*}
 \int_{B_\rho(x_0)} \mathcal{A}(\nabla w, \nabla \varphi) = 0
\end{align*}
for all $\varphi \in C^1_0(B_\rho(x_0),\setR^3)$.
\begin{lem}
 \label{lem:A-harmonic approximation}
 Consider a fixed positive $\beta>0$. Then given $\epsilon>0$ there exists a constant $C=C(\beta,\epsilon)>0$ with the following property: For every bilinear mapping $\mathcal{A}:\setR^{3\times3} \times \setR^{3\times3} \to \setR$ satisfying
 \begin{align}
 \label{eq:assumption A-harmonic}
  \mathcal{A}(B_1,B_1) \ge \frac{1}{\beta} \abs{B_1}^2 \, , \quad \abs{\mathcal{A}(B_1,B_2)}\le \beta \, \abs{B_1}\, \abs{B_2} \, , \quad B_1,B_2 \in \setR^{3\times3} \,,
 \end{align}
 for every $B_\rho(x_0) \subset \setR^3$, and every function $v \in H^1(B_\rho(x_0),\setR^3)$, there exists a $\mathcal{A}$-harmonic function $w \in  H^1(B_\rho(x_0),\setR^3)$ such that
\begin{align*}
 \int_{B_\rho(x_0)} \abs{\nabla w}^2 \le \int_{B_\rho(x_0)} \abs{\nabla v}^2
\end{align*}
and
\begin{align*}
 \bigg(\rho^{-3} \int_{B_\rho(x_0)} \abs{w - v}^2\bigg)^{\frac{1}{2}}
\le& C \sup \biggset{ \rho^{-1} \int_{B_\rho(x_0)} \mathcal{A}(\nabla v, \nabla \varphi) \, \bigg| \, \varphi \in C_0^\infty(B_\rho(x_0),\setR^3), \, \norm{\nabla \varphi}_{L^\infty} \le \rho^{-1}}
\\ 
&+ \epsilon \bigg( \rho^{-1} \int_{B_\rho(x_0)} \abs{\nabla v}^2 \bigg)^{\frac{1}{2}} \, .
\end{align*}
\end{lem}
\begin{proof}
 This is a special version of \cite[Lemma A.1]{duzaargrotowski}. We remark that the mentioned Lemma A.1 in \cite{duzaargrotowski} is -- according to the authors -- a reformulation of Lemma 2.1 in the same paper due to M. Giaquinta.
\end{proof}
We also recall the Campanato estimate for $\mathcal{A}$-harmonic functions (see \cite[III, Theorem 2.1]{giaquinta}):
\begin{lem}[Campanato estimate]
 \label{lem:Campanato estimates}
 Let $\mathcal{A}:\setR^{3 \times 3} \times \setR^{3 \times3} \to \setR$ be a bilinear mapping satisfying \nolinebreak \eqref{eq:assumption A-harmonic}. Then there is a constant $C>0$ depending only on $\beta$ such that 
\begin{align*}
 \int_{B_\rho(x_0)} \abs{w - \overline{w}_{x_0,\rho}}^2 \le C \Big(\frac{\rho}{R}\Big)^5 \int_{B_R(x_0)} \abs{w -  \overline{w}_{x_0,R}}^2
\end{align*}
for every $\mathcal{A}$-harmonic function $w$ on $B_R(x_0) \subset \setR^3$ and $0<\rho<R$.
\end{lem}
We are now prepared to prove the following lemma in the spirit of \cite{simon}.
\begin{lem}
  \label{thm:hoeldercontinuity}
 Suppose $u$ and $\A$ satisfy assumptions (A1) and (A2). Then for all $\Lambda>0$ and $\alpha \in (0,1)$ there exist positive constants $\delta_0,R_0$, and $C$ depending only on $\alpha$, $\Lambda$, and the typical parameters with the following property: If $B_{3r}(x_0) \subset \Omega$ is a ball with $0<r\le R_0$ such that
\begin{align*}
 \frac{1}{r^3} \int_{B_{2r}(x_0)} \abs{u - \ui_{x_0,2r}}^2 \le \delta_0^2 \qquad \text{and} \qquad
 \frac{1}{r} \int_{B_{3r}(x_0)} \abs{\nabla u \A}^2 \le \Lambda \, ,
\end{align*}
 then $u \in C^{0,\alpha}(B_{c_0 r}(x_0), \setR^3)$ and
  \begin{align*}
    \abs{u(x) - u(y)} \le C \max \biggset{\frac{1}{r^3} \int_{B_{2r}(x_0)} \abs{u- \ui_{x_0,2r}}^2 , \eta^2
      r^{2\alpha} }^{\frac{1}{2}} \bigg(\frac{\abs{x-y}}{c_0 r}\bigg)^{\alpha}
  \end{align*}
 for all $x,y \in B_{c_0r}(x_0)$.
\end{lem}
\begin{proof}
 Throughout the proof, $\Lambda>0$ and $\alpha \in (0,1)$ are fixed, and $B_{3r}(x_0)$ is a ball contained in $\Omega$ such that $\frac{1}{r} \int_{B_{3r}(x_0)} \abs{\nabla u \A}^2 \le \Lambda$. Furthermore, let $y \in B_r(x_0)$ and $0< \rho \le c_0 r$ be arbitrary, but fixed. In view of Lemma \ref{lem:A-harmonic approximation}, we consider a test function $\varphi \in C_0^\infty(B_{c_0 \frac{\rho}{2}}(y), \setR^3)$ such that 
$\norm{\nabla \varphi}_{L^\infty} \le 1 /(c_0 \rho / 2)$. This in particular implies that $\norm{\varphi}_{L^\infty} \le 2$.
By (A2) and the standard freezing coefficients device, we obtain
 \begin{align*}
     \int_{B_{c_0 \frac{\rho}{2}}(y)} \hspace{-0,1cm}\nabla u\A(y)\A(y)^T : \nabla \varphi
    =&\int_{B_{c_0 \frac{\rho}{2}}(y)} \hspace{-0,1cm}\nabla u \big(\A(y)\A(y)^T - \A\A^T \big): \nabla \varphi 
    +\int_{B_{c_0 \frac{\rho}{2}}(y)} \hspace{-0,1cm}\abs{\nabla u \A}^2 u \cdot \varphi 
    \\
    &+ \eta^2 \int_{B_{c_0 \frac{\rho}{2}}(y)} \hspace{-0,1cm}f \cdot \varphi 
   \\
   =&I_1+I_2+I_3 \,.
  \end{align*}
  The Lipschitz continuity of $\A\A^T$ and the H{\"o}lder inequality imply that
 \begin{align*}
  \abs{I_1} \le L \,c_0 \,\frac{\rho}{2} \,\norm{\nabla \varphi}_{L^\infty} \int_{B_{c_0 \frac{\rho}{2}}(y)} \abs{\nabla u} \le C \Big(c_0 \frac{\rho}{2} \Big)^{\frac{3}{2}} \bigg(\int_{B_{c_0 \frac{\rho}{2}}(y)} \abs{\nabla u}^2 \bigg)^{\frac{1}{2}}\,.
 \end{align*}
 Since $N=\sup_{x\in\Omega}\abs{\A(x)} < \infty$, we obtain for $I_2$ the estimate
 \begin{align*}
  \abs{I_2} \le \norm{\varphi}_{L^\infty} \int_{ B_{c_0 \frac{\rho}{2}}(y)} \abs{\nabla u \A }^2 \le 2 N^2 \int_{ B_{c_0 \frac{\rho}{2}}(y)} \abs{\nabla u}^2 \,.  
 \end{align*}
 Regarding $I_3$ we apply the H{\"o}lder inequality with $p = 3 /(2\alpha+1)$ in order to find
\begin{align*}
 \abs{I_3} \le \eta^2 \norm{\varphi}_{L^\infty} \int_{B_{c_0 \frac{\rho}{2}}(y)} \abs{f} \le 2 \eta^2 \abs{B_{c_0 \frac{\rho}{2}}(y)}^{\frac{2\alpha+1}{3}} \norm{f}_{L^{p'}} \le C \eta^2 \Big(c_0 \frac{\rho}{2} \Big)^{2\alpha+1} \,.
\end{align*}
A combination of the above estimates together with the reverse \Poincare{} inequality (see Corollary \ref{cor:poincare}) implies
\begin{align*}
 &\Big(c_0 \frac{\rho}{2}\Big)^{-1}\biggabs{\int_{B_{c_0 \frac{\rho}{2}}(y)} \nabla u\A(y)\A(y)^T : \nabla \varphi}
 \\
 \le& C \rho \bigg(\frac{1}{\rho^3} \int_{B_\rho(y)} \abs{u-\ui_{y,\rho}}^2 + \eta^2 \rho^{2\alpha} \bigg)^{\frac{1}{2}} + C \bigg(\frac{1}{\rho^3} \int_{B_\rho(y)} \abs{u-\ui_{y,\rho}}^2 + \eta^2 \rho^{2\alpha} \bigg)
\end{align*}
for all $\varphi \in C_0^\infty(B_{c_0 \frac{\rho}{2}}(y), \setR^3)$ with $\norm{\nabla \varphi}_{L^\infty} \le 1/(c_0 \rho/2)$, where $C$ is a constant depending only on $\alpha$, $\Lambda$, and the typical parameters. Let now $\epsilon >0$ be given (will be chosen later), and let
  $C_\epsilon>0$ be the corresponding constant from Lemma \ref{lem:A-harmonic approximation} with respect to the bilinear mapping $\mathcal{A}_y:\setR^{3\times3} \times \setR^{3\times3} \to \setR$ defined by
 $\mathcal{A}_y(B_1,B_2) = B_1 \A(y)\A(y)^T : B_2$
for $B_1,B_2 \in \setR^{3\times3}$. Thanks to (A1), the constant $C_\epsilon$ can be chosen independently of $y\in \Omega$. We find a $\mathcal{A}_y$-harmonic function
$w \in H^1\big(B_{c_0\frac{\rho}{2}}(y), \setR^3\big)$ such that
  \begin{align}
   \label{eq:Hoelder 1}
    \int_{B_{c_0 \frac{\rho}{2}}(y)} \abs{\nabla w}^2 \le \int_{B_{c_0 \frac{\rho}{2}}(y)} \abs{\nabla u}^2
  \end{align}
  and
  \begin{align*} \hspace{-0.1cm}
     \Big(c_0 \frac{\rho}{2} \Big)^{-3} \hspace{-0.1cm}\int_{B_{c_0 \frac{\rho}{2}}(y)} \hspace{-0.1cm}\abs{u - w}^2
     \le & C_\epsilon \, \rho^2 \bigg( \frac{1}{\rho^3} \hspace{-0.1cm}\int_{B_\rho (y)} \hspace{-0.1cm}\abs{u - \ui_{y,\rho}}^2 + \eta^2  \rho^{2 \alpha} \bigg)
     + C_\epsilon \, \bigg( \frac{1}{\rho^3}\hspace{-0.1cm} \int_{B_\rho (y)} \hspace{-0.1cm}\abs{u - \ui_{y,\rho}}^2 + \eta^2  \rho^{2 \alpha} \bigg)^{\hspace{-0.1cm}2}
     \\
     &+ \epsilon^2 \, \Big(c_0 \frac{\rho}{2} \Big)^{-1} \hspace{-0.1cm}\int_{B_{c_0 \frac{\rho}{2}}(y)}\hspace{-0.1cm} \abs{\nabla u}^{2} \,.
    \end{align*}
     With the help of the reverse \Poincare{} inequality, we obtain the estimate
    \begin{align*}
     \Big(c_0 \frac{\rho}{2} \Big)^{-3} \hspace{-0.1cm}\int_{B_{c_0 \frac{\rho}{2}}(y)} \hspace{-0.1cm}\abs{u - w}^2
     \le & C_\epsilon \, \rho^2 \bigg( \frac{1}{\rho^3} \hspace{-0.1cm}\int_{B_\rho (y)} \hspace{-0.1cm}\abs{u - \ui_{y,\rho}}^2 + \eta^2  \rho^{2 \alpha} \bigg)
     + C_\epsilon \, \bigg( \frac{1}{\rho^3}\hspace{-0.1cm} \int_{B_\rho (y)}\hspace{-0.1cm} \abs{u - \ui_{y,\rho}}^2 + \eta^2  \rho^{2 \alpha} \bigg)^{\hspace{-0.1cm}2}
     \\
     & + C  \epsilon^2 \bigg( \frac{1}{\rho^3} \hspace{-0.1cm}\int_{B_\rho(y)} \hspace{-0.1cm}\abs{u- \ui_{y,\rho}}^2 + \eta^2  \rho^{2 \alpha} \bigg) \,.
   \end{align*}
   For $\theta \in (0,1)$ (will be chosen later) we have together with \eqref{eq:mean value} that
  \begin{align*}
    \Big(\theta  c_0 \frac{\rho}{2} \Big)^{-3} \int_{B_{\theta  c_0 \frac{\rho}{2}}(y)} \abs{u- \ui_{y,\theta c_0 \frac{\rho}{2}}}^2 
    \le & 2 \Big(\theta  c_0 \frac{\rho}{2} \Big)^{-3} \int_{B_{\theta  c_0 \frac{\rho}{2}}(y)} \abs{u - w}^2 +
    2  \Big(\theta  c_0 \frac{\rho}{2} \Big)^{-3} \int_{B_{\theta  c_0 \frac{\rho}{2}}(y)} \abs{w - \overline{w}_{y,\theta c_0 \frac{\rho}{2}}}^2 \,.
  \end{align*}
  The Campanato estimate (see Lemma \ref{lem:Campanato estimates}) combined with the \Poincare{} inequality, \eqref{eq:Hoelder 1}, and the reverse \Poincare{} inequality implies
  \begin{align*}
    \Big(\theta  c_0 \frac{\rho}{2} \Big)^{-3} \int_{B_{\theta  c_0 \frac{\rho}{2}}(y)} \abs{w - \overline{w}_{y,\theta c_0 \frac{\rho}{2}}}^2  \le C  \theta^2 \bigg( \frac{1}{\rho^3} \int_{B_\rho (y)} \abs{u - \ui_{y,\rho}}^2 + \eta^2  \rho^{2 \alpha} \bigg) \,.
  \end{align*}
  A combination of the above estimates shows that
  \begin{align*}
    \Big(\theta  c_0 \frac{\rho}{2} \Big)^{-3}\hspace{-0.1cm} \int_{B_{\theta c_0 \frac{\rho}{2}}(y)} \hspace{-0.1cm}\abs{u- \ui_{y,\theta c_0 \frac{\rho}{2}}}^2 
     \le  \big( C_\epsilon\theta^{-3}\rho^2 + C_\epsilon \theta^{-3} I^2 + C  \theta^{-3} \epsilon^2 
     + C \theta^2 \big)I^2 \, ,
  \end{align*}
  where $I^2 = \max \bigset{\frac{1}{\rho^3} \int_{B_\rho(y)} \abs{u- \ui_{y,\rho}}^2 \, , \, \eta^2 \rho^{2 \alpha}}$.
  We introduce $ \kappa = \kappa(\theta)=  c_0 \theta /2 
  \in (0,c_0/2)$ and obtain
  \begin{align*}
    (\kappa \rho )^{-3} \int_{B_{\kappa \rho}(y)} \abs{u- \ui_{y,\kappa\rho}}^2
    \le & \big( C_\epsilon  \kappa^{-3}  \rho^2 + C_\epsilon \kappa^{-3} I^2 + C  \kappa^{-3}  \epsilon^2 
    + C  \kappa^2 \big) \, I^2 \, .
  \end{align*}
  We now choose $\kappa$ and $\epsilon$:
  We first choose $\kappa \in (0, c_0/2)$ such that $C \kappa^2 \le \kappa^{2 \alpha}/4$ and then choose $\epsilon>0$ such that $C \kappa^{-3} \epsilon^2 \le \kappa^{2 \alpha}/4$. This in particular implies that the constant $C_\epsilon$ is fixed. 
  We introduce the abbreviations $R_1 = \min \bigset{ \big(\kappa^{3 + 2 \alpha}/(4 C_\epsilon)\big)^{\frac{1}{2}} \, , 
   \, \big(\kappa^{3 + 2 \alpha}/(4 C_\epsilon) \big)^{\frac{1}{2 \alpha}}}$ and $\delta_1^2 = \kappa^{3 + 2 \alpha}/(4 C_\epsilon)$.
  If we assume that the radius $\rho$ is such that $\rho \le R_1$ and $\frac{1}{\rho^3} \int_{B_\rho(y)} \abs{u - \ui_{y,\rho}}^2 \le \delta_1^2$, then (w.l.o.g. $\eta\le1$)
  \begin{align*}
    (\kappa \rho)^{-3} \int_{B_{\kappa \rho}(y)} \abs{u - \ui_{y,\kappa\rho}}^2 \le \kappa^{2 \alpha} \max \biggset{ \frac{1}{\rho^3}
    \int_{B_\rho(y)} \abs{u - \ui_{y,\rho}}^2 \, , \, \eta^2\rho^{2 \alpha}} \, .
  \end{align*}
  We start an iteration process: Obviously, we have $\kappa  \rho \le R_1$ and
    $(\kappa \rho)^{-3} \int_{B_{\kappa \rho}(y)} \abs{u - \ui_{y,\kappa \rho}}^2 \le \delta_1^2$. 
  This means that the same assumptions are also satisfied for $y$ and $\kappa \rho$. 
  An induction argument shows that
  \begin{align*}
    \big(\kappa^k \rho\big)^{-3} \int_{B_{\kappa^k \rho}(y)} \abs{u - \ui_{y,\kappa^k \rho}}^2 \le \kappa^{2 k  \alpha}
    \max \biggset{ \frac{1}{\rho^3} \int_{B_\rho(y)} \abs{u - \ui_{y,\rho}}^2 \, , \, \eta^2 \rho^{2 \alpha}}
  \end{align*}
  for all $k \in \setN_0$, $y \in B_r(x_0)$, and $0<\rho \le c_0  r$, provided
    $\rho \le R_1$ and $\frac{1}{\rho^3} \int_{B_\rho(y)} \abs{u - \ui_{y,\rho}}^2 \le \delta_1^2$.
  We now define $R_0 = R_1/c_0$ and $\delta_0^2 = c_0^3 \delta_1^2$ and assume that
  \begin{align*}
    r \le R_0 \qquad \text{and} \qquad \frac{1}{r^3} \int_{B_{2r}(x_0)} \abs{u - \ui_{x_0,2r}}^2 \le \delta_0^2 \,.
  \end{align*}
  In particular, the above estimate is true for every choice of $y \in B_r(x_0)$ and $\rho = c_0 r$. For a given $\sigma \in (0,c_0 r]$ there is a 
  $k \in \setN_0$ such that $\kappa^{k+1} c_0 r \le \sigma \le \kappa^k c_0 r$,
  hence
  \begin{align*}
    \frac{1}{\sigma^3} \int_{B_\sigma(y)} \abs{u- \ui_{y,\sigma}}^2 \le C\Big( \frac{\sigma}{c_0 r}\Big)^{2 \alpha} 
    \max \biggset{ \frac{1}{r^3} \int_{B_{2r}(x_0)} \abs{u - \ui_{x_0,2r}}^2 \, , \, \eta^2 r^{2 \alpha}} \,.
  \end{align*}
  An application of the Campanato lemma (see Lemma \ref{lem:Campanato}) yields the desired result.
\end{proof}

\section{Higher regularity}
\label{sec:Higher regularity}
With the help of Lemma \ref{thm:hoeldercontinuity}, we can now prove the small-energy-regularity theorem:
\begin{theorem}
  \label{thm:linftygradient}
  Suppose $u$ and $\A$ satisfy assumptions (A1) and (A2). For every $\Lambda>0$ there are positive constants $\delta_0$, $R_0$, and $C$ depending only on $\Lambda$ and the typical parameters with the following property:
  If $B_{3r}(x_0) \subset \Omega$ is a ball with $0<r\le R_0$ such that
  \begin{align*}
  \frac{1}{r^3} \int_{B_{2r}(x_0)} \abs{u - \ui_{x_0,2r}}^2 \le \delta_0^2 \qquad \text{and} \qquad
  \frac{1}{r} \int_{B_{3r}(x_0)} \abs{\nabla u \A}^2 \le \Lambda \, ,
  \end{align*}
  then $u \in H^2(B_{\theta r}(x_0), \setR^3) \cap C^{1,\gamma}(\overline{B_{\theta r}(x_0)}, \setR^3)$ for every $\gamma \in (0,1)$, where the constant $\theta \in (0,1)$ depends only on $c_0$. Moreover, we have the estimate
  \begin{align*}
    r \hspace{-0.1cm} \sup_{B_{\theta r}(x_0)} \abs{\nabla u} \le C \max \biggset{ \frac{1}{r^3} \int_{B_{2r}(x_0)} \abs{u - \ui_{x_0,2r}}^2 \, , \, 
    \eta^2  r}^{\frac{1}{2}} \,.
  \end{align*}
\end{theorem}
\begin{proof}
 Let $\Lambda>0$ be given and define $\alpha = 11 / 12$. We find positive constants $\delta_0$, $R_0$, and $C$ as in Lemma \ref{thm:hoeldercontinuity} with respect to $\Lambda$ and $\alpha$. Let now $B_{3r}(x_0) \subset \Omega$ be a ball with radius $0<r\le R_0$ such that $\frac{1}{r^3} \int_{B_{2r}(x_0)} \abs{u - \ui_{x_0,2r}}^2 \le \delta_0^2$ and $\frac{1}{r} \int_{B_{3r}(x_0)} \abs{\nabla u \A}^2 \le \Lambda$. This in particular implies that $u \in C^{0,\alpha}(B_{c_0 r}(x_0) , \setR^3)$ and
 \begin{align}
  \label{eq:differentiability 1}
  \abs{u(x) - u(y)} \le C  I_0 \, \bigg( \frac{\abs{x-y}}{c_0 r} \bigg)^{\alpha}
 \end{align}
  for every $x,y \in B_{c_0 r}(x_0)$, where
    $I_0 = \max \set{ \frac{1}{r^3} \int_{B_{2r}(x_0)} \abs{u - \ui_{x_0,2r}}^2 \, , \, \eta^2  r^{2 \alpha}}^{\frac{1}{2}}$.
  In the sequel we show that $\nabla u$ belongs to $C^{0,\gamma}$ for $\gamma = 6 / 71$ on a ball centered at $x_0$. Let therefore $y \in B_{c_0 \frac{r}{2}} (x_0) $ and $0 < \rho \le 
  c_0^2 r/4$ be given and consider the unique solution $w \in H^1(B_\rho(y) , \setR^3)$ of
  \begin{align*}
    \Div \big(\nabla w \A(y)\A(y)^T\big) &= 0 \qquad \text{in } B_\rho(y)
    \\
    w &= u \qquad \text{on } \partial B_\rho(y) \,.
  \end{align*}
  An application of \cite[III, Proposition 2.3]{giaquinta} implies the 
  estimate
  \begin{align}
   \label{eq:differentiability 2}
    \sup_{B_\rho (y)} \abs{w} \le C  \hspace{-0.1cm}\sup_{\partial B_\rho (y)} \abs{u} = C\, ,
  \end{align}
  where, thanks to (A1), the constant $C$ is independent of the special 
 choice of $y \in \Omega$. In particular, we have $w \in H^1\big(B_\rho (y), \setR^3\big) \cap L^\infty \big(B_\rho (y), \setR^3\big)$, and we obtain with the help of assumption (A2) that
  \begin{align*}
    & \int_{B_\rho(y)} \hspace{-0.1cm}(\nabla u -\nabla w ) \A(y)\A(y)^T : \nabla \varphi
    \\
    =& \int_{B_\rho(y)} \hspace{-0.1cm}\nabla u \big( \A(y)\A(y)^T - \A\A^T \big) : \nabla \varphi + \int_{B_\rho(y)} \hspace{-0.1cm}\abs{\nabla u \A}^2 u \cdot \varphi + \eta^2 \int_{B_\rho(y)} \hspace{-0.1cm}f \cdot \varphi
  \end{align*}
  for every test function $ \varphi \in C_0^\infty(B_\rho (y), \setR^3)$. Moreover, it is easily seen that
    $\varphi = u - w$ belongs to the function space $H^1_0(B_\rho (y), \setR^3) \cap
  L^ \infty(B_\rho (y), \setR^3)$ and therefore
 is an admissible test function.
 This implies together with (A1) and the Young inequality that
 \begin{align*}
    \frac{1}{\beta} \int_{B_\rho (y)} \abs{\nabla u- \nabla w}^2
   \le & C \rho^2 \int_{B_\rho(y)} \abs{\nabla u}^2 + \frac{1}{2\beta} \int_{B_\rho(y)} \abs{\nabla u - \nabla w}^2
   + N^2 \sup_{B_\rho(y)} \abs{u -w} \int_{B_\rho(y)} \abs{\nabla u}^2 
   \\
   &+  \eta^2 \sup_{B_\rho(y)} \abs{u-w} 
   \int_{B_\rho(y)} \abs{f} \, .
 \end{align*}
 Now, we absorb the term $\frac{1}{2\beta} \int_{B_\rho(y)} \abs{\nabla u - \nabla w}^2$ on the left hand side and apply the H{\"o}lder inequality with $p = 3 /(2\alpha+1)$ to the last term on the right hand side in order to find the following estimate:
 \begin{align*}
  \frac{1}{2\beta} \int_{B_\rho (y)} \abs{\nabla u- \nabla w}^2
  \le& C \rho^2 \int_{B_\rho(y)} \abs{\nabla u}^2 + N^2 \sup_{B_\rho(y)} \abs{u -w} \int_{B_\rho(y)} \abs{\nabla u}^2 
   +  C \eta^2 \sup_{B_\rho(y)} \abs{u-w} \,\rho^{2\alpha+1} \, .
 \end{align*}
 For every $x \in B_\rho(y)$ and arbitrary (but fixed) $z \in \partial B_\rho (y)$, we obtain with the help of
 \eqref{eq:differentiability 1} and an estimate similar to \eqref{eq:differentiability 2} that
 \begin{align*}
   \abs{u(x) - w(x)} & \le \abs{u(x) - u(z)} + \abs{u(z) - w(x)}
   \le C  I_0  \Big( \frac{\rho}{c_0 r} \Big)^\alpha + C \hspace{-0.1cm} \sup_{\partial B_\rho (y)} \abs{u(z) - u}
   \le C I_0 \Big( \frac{\rho}{c_0 r} \Big)^\alpha \,.
 \end{align*}
 The reverse \Poincare{} inequality (see Corollary \ref{cor:poincare}) and 
the easily verified fact that $\rho^2 \le C(\rho/ c_0r)^\alpha$ imply the 
 estimate
 \begin{align*}
   \int_{B_\rho (y)} \abs{\nabla u - \nabla w}^2
 \le & C \Big(\frac{\rho}{c_0r}\Big)^\alpha \Big(\frac{2\rho}{c_0}\Big)^{-2} \int_{B_{\frac{2\rho}{c_0}}(y)} \abs{u - \ui_{y,\frac{2\rho}{c_0}}}^2
 + C \eta^2 \Big(\frac{\rho}{c_0r}\Big)^\alpha \rho^{2\alpha+1}
\\
&+ CI_0 \Big(\frac{\rho}{c_0r}\Big)^\alpha \Big(\frac{2\rho}{c_0}\Big)^{-2} \int_{B_{\frac{2\rho}{c_0}}(y)} \abs{u - \ui_{y,\frac{2\rho}{c_0}}}^2 + C \eta^2 I_0 \Big(\frac{\rho}{c_0r}\Big)^\alpha \rho^{2\alpha+1}\,.
 \end{align*}
With the help of the Jensen inequality and \eqref{eq:differentiability 1}, we obtain that
 \begin{align}
  \label{eq:differentiability 4}
  \begin{split}
   \int_{B_{\frac{2 \rho}{c_0}} (y)} \abs{u - \ui_{y,\frac{2 \rho}{c_0}}}^2
    \le \int_{B_{\frac{2 \rho}{c_0}} (y)} 
   \dashint\nolimits_{\hspace{-0.16cm}B_{\frac{2 \rho}{c_0}} (y)} \abs{u(x) - u(z)}^2 \, dz\, dx 
    \le C  I_0^2  \rho^3 \Big( \frac{\rho}{c_0 r} \Big)^{2 \alpha} \, .
  \end{split}
 \end{align}
 We can assume $I_0 \le 1$, hence
 \begin{align*}
   \frac{1}{\rho^3} \int_{B_\rho(y)} \abs{\nabla u - \nabla w}^2
   \le & C  I_0^2 (c_0 r)^{-2} \Big( \frac{ \rho}{c_0r} \Big)^{3 \alpha -2} + C \eta^2 (c_0 r)^{2 \alpha -2} 
   \Big( \frac{ \rho}{c_0r} \Big)^{3 \alpha -2} \, .
 \end{align*}
 This together with $\eta^2 (c_0 r )^{2 \alpha} \le c_0^{2 \alpha} I_0^2$ implies that
 \begin{align}
 \label{eq:differentiability 3}
   &\frac{1}{\rho^3} \int_{B_\rho(y)} \abs{\nabla u - \nabla w}^2
   \le CI_0^2 (c_0 r)^{-2} \Big( \frac{ \rho}{c_0r} \Big)^{3 \alpha -2} \,.
 \end{align}
 For $0<\sigma < \rho$ we find with the help of \eqref{eq:mean value} and the Campanato estimate (see Lemma \ref{lem:Campanato estimates}) that
 \begin{align*}
    \frac{1}{\sigma^3} \int_{B_\sigma(y)} \abs{\nabla u - {\overline{\nabla u}}_{y,\sigma}}^2  
   \le & 2 \Big( \frac{\rho}{\sigma}\Big)^3 \frac{1}{\rho^3} \int_{B_\rho(y)}\abs{\nabla u - \nabla w}^2
   + C \Big( \frac{\sigma}{\rho}\Big)^2 \frac{1}{\rho^3} \int_{B_\rho(y)} \abs{\nabla w - {\overline{\nabla w}}_{y,\rho}}^2
   \\
   \le & C I_0^2 (c_0 r)^{-2} \Big( \frac{\rho}{c_0 r} \Big)^{3 \alpha -2} \Big( \frac{\rho}{\sigma} \Big)^3 
   + C \Big( \frac{\sigma}{\rho}\Big)^2 \frac{1}{\rho^3} \int_{B_\rho(y)} \abs{\nabla w - {{\overline{\nabla w}}_{y,\rho}}}^2 \, .
 \end{align*}
 We estimate the remaining integral on the right hand side with the help of \eqref{eq:mean value} (choose $\lambda=0$),  \eqref{eq:differentiability 3}, the reverse \Poincare{} inequality (see Corollary \ref{cor:poincare}), \eqref{eq:differentiability 4}, and the fact that $\eta^2 (c_0 r )^{2 \alpha} \le c_0^{2 \alpha} I_0^2$ as follows:
 \begin{align*}
    \frac{1}{\rho^3} \int_{B_\rho (y)} \abs{\nabla w - {\overline{\nabla w}}_{y,\rho}  }^2
   \le & \frac{2}{\rho^3} \int_{B_\rho (y)} \abs{\nabla w - \nabla u}^2 + \frac{2}{\rho^3} \int_{B_\rho (y)} \abs{\nabla u}^2
   \\
   \le & C  I_0^2 (c_0 r)^{-2} \Big( \frac{\rho}{c_0r} \Big)^{3 \alpha -2} \hspace{-0.1cm}
   + C  I_0^2 (c_0 r)^{-2} \Big(\frac{\rho}{c_0 r} \Big)^{2 \alpha -2} \hspace{-0.1cm}
   + C  \eta^2 (c_0 r)^{2 \alpha - 2}\Big(\frac{\rho}{c_0 r} \Big)^{2 \alpha -2}
   \\
   \le & C  I_0^2 (c_0 r)^{-2} \Big(\frac{\rho}{c_0 r} \Big)^{2 \alpha -2} \, .
 \end{align*}
 We conclude 
 \begin{align*}
   \frac{1}{\sigma^3} \hspace{-0.1cm}\int_{B_\sigma(y)} \hspace{-0.1cm}\abs{\nabla u - {\overline{\nabla u}}_{y,\sigma}}^2 \le C I_0^2 (c_0r)^{-2} \bigg( \hspace{-0.1cm}\Big(\frac{\rho}{c_0r}\Big)^
   {3 \alpha -2} \Big(\frac{\rho}{\sigma} \Big)^3 \hspace{-0.1cm}+  \Big(\frac{\rho}{c_0r}\Big)^{2 \alpha -2} \Big( \frac{\sigma}{\rho} \Big)^2 \bigg)
 \end{align*}
 for all $y \in B_{c_0\frac{r}{2}}(x_0)$ and $ 0 < \sigma < \rho \le c_0^2 r / 4$. We now define $ \kappa =1 + \alpha / 5$ and
   $\rho(\sigma) = (\sigma /{c_0 r})^{\frac{1}{\kappa}} c_0 r$ for $0<\sigma \le c_0^2 
   (c_0/4)^{\kappa -1} r / 4$.
 This in particular implies that $ 0<\sigma < \rho(\sigma) \le c_0^2 r/4$, and we obtain by the above estimate
 \begin{align*}
   \frac{1}{\sigma^3} \int_{B_\sigma(y)} \abs{\nabla u - {\overline{\nabla u}}_{y,\sigma}}^2 \le C  I_0^2 (c_0r)^{-2} \Big(\frac{\sigma}{c_0r} \Big)^{2\gamma}
 \end{align*}
 for all $y \in B_{c_0 \frac{r}{2}}(x_0)$ and $0<\sigma \le c_0^2 
 (c_0/4)^{\kappa - 1} r/4$,
 where $\gamma = 6/71$. Therefore, the Campanato lemma (see Lemma \ref{lem:Campanato}) implies that $\nabla u \in C^{0,\gamma} 
 ( B_{3\theta r}(x_0), \setR^{3\times3})$ with $3\theta = c_0^2(c_0/4)^{\kappa - 1}/4$. Moreover, we have the estimate
 \begin{align*}
   \abs{\nabla u(x) - \nabla u(y)} \le C I_0 r^{-1} \bigg( \frac{\abs{x-y}}{r} \bigg)^\gamma
 \end{align*}
 for all $x,y \in B_{3\theta r}(x_0)$. For a given $x \in B_{3\theta r}(x_0)$, we can now estimate the derivative $\nabla u$ of $u$ as follows:
 \begin{align*}
   \abs{\nabla u(x)} 
    \le \dashint\nolimits_{\hspace{-0.16cm}B_{3\theta r}(x_0)} \abs{\nabla u(x) - \nabla u(y)} \, dy + C  (\theta r)^{-3} \int_{B_{3\theta r}(x_0)} \abs{\nabla u} 
    \le C I_0 r^{-1} + C r^{-1} \bigg( \frac{1}{r} \int_{B_{3\theta r}(x_0)} \abs{\nabla u}^2 \bigg)^{\frac{1}{2}} .
 \end{align*}
 We obtain with the help of the reverse \Poincare{} inequality (see Corollary \ref{cor:poincare}) that
 \begin{align*}
   \abs{\nabla u(x)} & \le C I_0 r^{-1} + C r^{-1} \bigg( \frac{1}{r^3} \int_{B_{2 r}(x_0)} \abs{u- \ui_{x_0,2 r}}^2 + \eta^2 r^{2 \alpha} \bigg)^{\frac{1}{2}}
   \le C I_0 r^{-1} \,,
 \end{align*}\
 hence ($2\alpha>1$ and $R_0 \le 1$)
 \begin{align*}
   r \hspace{-0.1cm}\sup_{B_{3\theta r}(x_0)} \abs{\nabla u} \le C \max \biggset{ \frac{1}{r^3} \int_{B_{2r}(x_0)} \abs{u -\ui_{x_0,2r}}^2 \, , \, \eta^2  r}^
   {\frac{1}{2}} \, .
 \end{align*}
 Moreover, $u$ is a weak solution of
 $\Div (\nabla u \A\A^T) = F$ in $B_{3\theta r}(x_0)$
 with Lipschitz continuous coefficients $\A\A^T$ and right hand side $F \in L^p$ for every $1<p<\infty$. With the help of the interior $L^2$-regularity theory for linear elliptic systems, we obtain that $u \in H^2(B_{2\theta r}(x_0), \setR^3)$ (see \cite{GiaquintaII}). 
Now, we can apply standard $L^p$-estimates for elliptic systems in non-divergence form and find $u \in W^{2,p}(B_{\theta r}(x_0), \setR^3)$ for every $1<p<\infty$. The Sobolev embedding theorem implies that $u \in C^{1,\gamma}(\overline{B_{\theta r}(x_0)}, \setR^3)$ for every $\gamma \in (0,1)$, and the theorem is proved.
\end{proof}
Finally, we apply the small-energy-regularity theorem in combination with a covering argument to obtain our main result:
\begin{proof}[Proof (of the main result).]
 Let $u_\eta$ be a minimizer of $E^\eta$ with parameter $0<\eta \le 1$. We use the constant comparison function $v = e_1$ to find that
  \begin{align}
   \label{eq:regularitysmall}
    \int_\Omega \abs{\nabla u_\eta}^2 \le E^\eta(u_\eta) \le E^\eta(e_1) = \eta^2 \int_{\setR^3} \abs{H[e_1]}^2 \le C  \eta^2 \,.
  \end{align}
The idea is to use this estimate to guarantee for sufficiently small $\eta$ the smallness assumption of Theorem \nolinebreak \ref{thm:linftygradient}. Let now $x \in \overline{\Omega}$ be an arbitrary point. If $x \in \Omega$, then $u_\eta$ satisfies assumptions (A1) and (A2) on $\Omega$ with $\A\equiv I$ and
$f = H[u_\eta] -  u_\eta \cdot H[u_\eta] \, u_\eta$
thanks to Lemmas \ref{lem:eulerlagrange} and \ref{lem:almostminimizer}. The monotonicity formula (see Lemma \ref{lem:monotonicity}) and \eqref{eq:regularitysmall} imply that
 $\frac{1}{t} \int_{B_{t}(x)} \abs{\nabla u_\eta}^2 
     \le C(x) \eta^2$
for every $0<t\le 3d(x)$, where $3d(x) =  \text{dist}(x,\partial \Omega)$. We define $\Lambda(x) = 3 C(x)$ and obtain positive $\delta_0(x)$, $R_0(x)$, and $\theta(x) \in (0,1)$ as in Theorem \ref{thm:linftygradient}. For
$r(x) = \min \set{d(x) \, , \, R_0(x)}$
we have that $r(x) \le R_0(x)$, $\frac{1}{r(x)} \int_{B_{3r(x)}(x)} \abs{\nabla u_\eta}^2 \le \Lambda(x)$,
and thanks to the \Poincare{} inequality, we obtain
\begin{align*}
  \frac{1}{r(x)^3} \int_{B_{2 r(x)}(x)} \abs{u_\eta - {\overline{u_\eta}}_{x,2r(x)}}^2 \le \frac{C}{r(x)} \int_{B_{2r(x)}(x)} \abs{\nabla u_\eta}^2 \le C_0(x) 
    \eta^2 \,.
\end{align*}
If $x$ belongs to the boundary of $\Omega$, then we have to use the reflection method as described in Section \ref{sec:reflection method}. More precisely, we flatten the boundary of $\Omega$ with the help of the $C^1$-diffeomorphism $\psi_x:U_x \to B_{R_x}$ from Lemma \ref{lem:coordinates} and define the function $\overline{u}_\eta^x \in H^1(B_{R_x},S^2)$ by reflection. Here, $U_x$ is an open neighborhood of $x$ in $\setR^3$, $\psi_x(x) = 0$, $\psi_x(\Omega\cap U_x) = B_{R_x}^+$, and $u_\eta = \overline{u}_\eta^x \circ \psi_x \quad \text{on } \Omega 
\cap U_x$.
Moreover, $\overline{u}_\eta^x$ satisfies assumptions (A1) and (A2) on $B_{R_x}$ with $\A$ defined as in Lemma \ref{lem:coordinates} and $f = \overline{H}[\overline{u}_\eta^x] - \overline{u}_\eta^x \cdot \overline{H}[\overline{u}_\eta^x] \,\overline{u}_\eta^x$
(see Lemmas \ref{lem:eulerlagrangeloc} and \ref{lem:almostminimizerloc}).
As above, the monotonicity formula (see Lemma \ref{thm:monotonicity}) combined with the transformation formula and \eqref{eq:regularitysmall} implies that
\begin{align*}
 \frac{1}{t}\int_{B_t} \abs{\nabla \overline{u}_\eta^x \A}^2 \le C(x) \eta^2
\end{align*}
for every $0<t\le 3 d(x)$, where $d(x)$ is chosen sufficiently small. For $\Lambda(x) = 3 C(x)$ we obtain positive 
$\delta_0(x)$, $R_0(x)$, and $\theta(x) \in (0,1)$ as in Theorem \ref{thm:linftygradient}. Moreover, $r(x)$ defined by $r(x) = \min\bigset{d(x), R_0(x)}$
satisfies
  $r(x) \le R_0(x)$, $\frac{1}{r(x)} \int_{B_{3 r(x)}} \abs{\nabla \overline{u}_\eta^x \A}^2 \le \Lambda(x)$,
and
\begin{align*}
 \frac{1}{r(x)^3} \int_{B_{2  r(x)}} \abs{\overline{u}_\eta^x -\overline{\hspace{0.05cm}\overline{u}_\eta^x}_{\,0,2r(x)}\hspace{0.05cm}}^2  
   \le C_0(x) \eta^2 \, ,
\end{align*}
where we have used the \Poincare{} inequality and $\sup_{y \in B_{R_x}} \abs{\A(y)^{-1}} < \infty$. Furthermore, we know that
\begin{align*}
 \overline{\Omega} \subset \bigcup_{x \in \Omega} B_{\theta(x)r(x)}(x) \cup  \bigcup_{x \in \partial \Omega} \psi_x^{-1}
    \big( B_{\theta(x) r(x)} \big) \, .
\end{align*}
Due to the compactness of $\overline{\Omega}$, there exist finitely many points 
$x_1,\dots,x_n \in \Omega$ and $x_{n+1},\dots,x_{n+m} \in \partial \Omega$
such that
\begin{align*}
  \overline{\Omega} \subset \bigcup_{i=1}^n B_{\theta(x_i)r(x_i)}(x_i) \cup  \bigcup_{j=1}^m \psi_{x_{n+j}}^{-1}
    \big( B_{\theta(x_{n+j}) r(x_{n+j})} \big) \, .
\end{align*}
We define $C_1(\Omega)  = \max \set{ C_0(x_1) , \dots ,C_0(x_{n+m})}$, $\delta_0(\Omega)  = \min \set{\delta_0(x_1) , \dots ,\delta_0(x_{n+m})}$,
and 
$\eta^2_0(\Omega) = \delta_0(\Omega)^2/C_1(\Omega)$.
An application of Theorem \ref{thm:linftygradient} now shows that every minimizer $u_\eta$ of $E^\eta$ with parameter $0<\eta\le\eta_0$ belongs to $H^2(\Omega,\setR^3) \cap C^{1,\gamma}(\overline{\Omega},\setR^3)$ for every $\gamma \in (0,1)$ and $\norm{\nabla u_\eta}_{L^\infty} \le C_0(\Omega) \eta$.
Moreover, Lemma \ref{lem:eulerlagrange} implies that 
\begin{align*}
 \int_\Omega \nabla u_\eta : \nabla \varphi = \int_\Omega g \cdot \varphi
\end{align*}
for all $\varphi \in C^1_0(\setR^3,\setR^3)$, where $g \in L^2(\Omega,\setR^3)$ is a function. We conclude $\frac{\partial u_\eta}{\partial \nu} = 0$ on $\partial \Omega$. This completes the proof.
\end{proof}

\medskip

\noindent {\bf Acknowledgments.}
This work is part of the author's PhD thesis prepared at the Max Planck Institute for Mathematics in the Sciences (MPIMiS) and submitted in June 2009 at the University of Leipzig, Germany. The author would like to thank his supervisor Stefan M{\"u}ller for the opportunity to work at MPIMiS and for having chosen an interesting problem to work on. Financial support from the International Max 
Planck Research School `Mathematics in the Sciences' (IMPRS) is also acknowledged.

\bibliographystyle{abbrv}
\bibliography{article-regularity}

\bigskip\small

\noindent{\sc NWF I-Mathematik, Universit\"at Regensburg,  93040 Regensburg}\\
{\it E-mail address}: {\tt alexander2.huber@mathematik.uni-regensburg.de}

\end{document}